\pdfoutput=1
\documentclass[11pt,reqno]{amsart}
\usepackage[letterpaper,margin=1in,footskip=0.25in]{geometry}
\usepackage{microtype}
\usepackage[only,llbracket,rrbracket]{stmaryrd}
\usepackage{enumitem}
\usepackage{amsmath}
\usepackage{amssymb}
\usepackage{mathtools}
\usepackage{tikz-cd}
\usepackage[justification=centering]{caption}
\usepackage[noadjust]{cite}
\renewcommand{\citepunct}{;\penalty\citemidpenalty\ }

\PassOptionsToPackage{pdfusetitle,pagebackref,colorlinks,bookmarksdepth=3}{hyperref}
\usepackage{bookmark}
\hypersetup{citecolor=[HTML]{2E7E2A},linkcolor=[HTML]{800006},urlcolor=[HTML]{8A0087}}

\newtheorem{theorem}{Theorem}[section]
\newtheorem{lemma}[theorem]{Lemma}
\newtheorem{proposition}[theorem]{Proposition}

\newtheorem{innercustomthm}{Theorem}
\newenvironment{customthm}[1]
  {\renewcommand\theinnercustomthm{#1}\innercustomthm}
  {\endinnercustomthm}

\theoremstyle{definition}

\newtheorem{setup}[theorem]{Setup}

\theoremstyle{remark}
\newtheorem{remark}[theorem]{Remark}

\newtheoremstyle{cited}{.5\baselineskip\@plus.2\baselineskip\@minus.2\baselineskip}{.5\baselineskip\@plus.2\baselineskip\@minus.2\baselineskip}{\itshape}{}{\bfseries}{\bfseries .}{5pt plus 1pt minus 1pt}{\thmname{#1}\thmnumber{ #2}\thmnote{ \normalfont#3}}
\theoremstyle{cited}
\newtheorem{citedthm}[theorem]{Theorem}
\newtheorem{citedques}[theorem]{Question}
\newtheorem{cancellationprob}[theorem]{Cancellation Problem}

\newtheoremstyle{citeddef}{.5\baselineskip\@plus.2\baselineskip\@minus.2\baselineskip}{.5\baselineskip\@plus.2\baselineskip\@minus.2\baselineskip}{}{}{\bfseries}{\bfseries .}{5pt plus 1pt minus 1pt}{\thmname{#1}\thmnumber{ #2}\thmnote{ \normalfont#3}}
\theoremstyle{citeddef}
\newtheorem{citeddef}[theorem]{Definition}

\newtheoremstyle{step}{.25\baselineskip\@plus.1\baselineskip\@minus.1\baselineskip}{.25\baselineskip\@plus.1\baselineskip\@minus.1\baselineskip}{\itshape}{}{\bfseries}{\bfseries .}{5pt plus 1pt minus 1pt}{\thmname{#1}\thmnumber{ #2}\thmnote{ \normalfont(#3)}}
\theoremstyle{step}
\newtheorem{step}{Step}[theorem]

\DeclareMathOperator{\Min}{Min}
\DeclareMathOperator{\Spec}{Spec}
\DeclareMathOperator{\height}{ht}

\newcommand{\Z}{\mathbb{Z}}
\newcommand{\Q}{\mathbb{Q}}

\newcommand{\fm}{\mathfrak{m}}
\newcommand{\fp}{\mathfrak{p}}
\newcommand{\fq}{\mathfrak{q}}
\newcommand{\red}{\mathrm{red}}

\providecommand\given{}
\newcommand\SetSymbol[1][]{\nonscript\:#1\vert\allowbreak\nonscript\:\mathopen{}}
\DeclarePairedDelimiterX\Set[1]\{\}{\renewcommand\given{\SetSymbol[\delimsize]}#1}

\tikzcdset{scale cd/.style={every label/.append style={scale=#1},cells={nodes={scale=#1}}}}

\begin{document}
\title{Variants of normality and steadfastness deform}
\thanks{This research was initiated as part of the
\href{https://sites.google.com/umich.edu/mreg2021}{2021 Michigan Research
Experience for Graduates}, which was (partially) supported by the
Rackham Graduate School through its
\href{https://rackham.umich.edu/faculty-and-staff/faculty-and-program-funding/rackham-faculty-allies-and-student-ally-diversity-grants/\#rfadg}{Faculty Allies for Diversity Program}.}

\author[A. Bauman]{Alexander Bauman}
\address{Department of Mathematics\\University of Michigan\\Ann Arbor, MI 48109-1043\\USA}
\email{\href{mailto:adbauman@umich.edu}{adbauman@umich.edu}}

\author[H. Ellers]{Havi Ellers}
\address{Department of Mathematics\\University of Michigan\\Ann Arbor, MI 48109-1043\\USA}
\email{\href{mailto:ellers@umich.edu}{ellers@umich.edu}}

\author[G. Hu]{Gary Hu}
\address{Department of Mathematics\\University of Michigan\\Ann Arbor, MI 48109-1043\\USA}
\email{\href{mailto:garyjhu@umich.edu}{garyjhu@umich.edu}}

\author[T. Murayama]{\\Takumi Murayama}
\address{Department of Mathematics\\Purdue University\\West Lafayette, IN
47907-2067\\USA}
\email{\href{mailto:murayama@purdue.edu}{murayama@purdue.edu}}
\thanks{Murayama was supported by the National Science
Foundation under Grant No.\ DMS-1902616}

\author[S. Nair]{Sandra Nair}
\address{Department of Mathematics\\University of Michigan\\Ann Arbor, MI 48109-1043\\USA}
\email{\href{mailto:sdnair@umich.edu}{sdnair@umich.edu}}

\author[Y. Wang]{Ying Wang}
\address{Department of Mathematics\\University of Michigan\\Ann Arbor, MI 48109-1043\\USA}
\email{\href{mailto:ywangx@umich.edu}{ywangx@umich.edu}}



\makeatletter
\hypersetup{
  pdfauthor={Alexander Bauman, Havi Ellers, Gary Hu, Takumi Murayama, Sandra Nair, and Ying Wang},
  pdfsubject=\@subjclass,pdfkeywords=\@keywords
}
\makeatother

\begin{abstract}
  The cancellation problem asks whether $A[X_1,X_2,\ldots,X_n]
  \cong B[Y_1,Y_2,\ldots,Y_n]$ implies $A \cong B$.
  Hamann introduced the class of steadfast rings as the rings for which
  a version of the cancellation problem considered by Abhyankar,
  Eakin, and Heinzer holds.
  By work of Asanuma, Hamann, and Swan, steadfastness can be characterized in
  terms of $p$-seminormality, which is a variant of normality introduced by
  Swan.
  We prove that $p$-seminormality and steadfastness deform for reduced
  Noetherian local rings.
  We also prove that $p$-seminormality and steadfastness are stable under
  adjoining formal power series variables for reduced (not necessarily Noetherian) rings.
  Our methods also give new proofs of the facts that normality and weak normality
  deform, which are of independent interest.
\end{abstract}

\maketitle

\section{Introduction}
Let $A$ and $B$ be rings, which we will always assume to be commutative with
identity.
If $A$ is isomorphic to $B$, then the polynomial rings $A[X_1,X_2,\ldots,X_n]$
are isomorphic to $B[Y_1,Y_2,\ldots,Y_n]$ for all $n \ge 1$.
Coleman and Enochs \cite{CE71} and Abhyankar, Eakin, and Heinzer \cite{AEH72}
asked whether the converse holds.
\begin{cancellationprob}[{\citeleft\citen{CE71}\citemid p.\ 247\citepunct
\citen{AEH72}\citemid Question 7.3\citeright}]\label{ques:cancellation}
  Let $A$ and $B$ be rings, and let $X_1,X_2,\ldots,X_n$ and
  $Y_1,Y_2,\ldots,Y_n$ be indeterminates.
  If
  \[
    A[X_1,X_2,\ldots,X_n] \cong B[Y_1,Y_2,\ldots,Y_n]
  \]
  for some $n \ge 1$, do we have $A \cong B$?
\end{cancellationprob}
In \cite{Hoc72}, Hochster showed that the Cancellation Problem
\ref{ques:cancellation} has a negative answer.
On the other hand,
in \cite[(4.5)]{AEH72}, Abhyankar, Eakin, and Heinzer showed that
a variant of the Cancellation Problem
\ref{ques:cancellation} holds for unique factorization domains.
Motivated in part by this latter result, Hamann considered the following version
of the Cancellation Problem \ref{ques:cancellation} in \cite{Ham75}.
\begin{citedques}[{\citeleft\citen{Ham75}\citemid p.\ 1\citepunct
\citen{Hoc75}\citemid p.\ 63\citeright}]
  \label{ques:hamann}
  Let $A \subseteq B$ be a ring extension, and let $X,X_1,X_2,\ldots,X_n$ and
  $Y_1,Y_2,\ldots,Y_n$ be indeterminates.
  If
  \[
    A[X,X_1,X_2,\ldots,X_n] \cong B[Y_1,Y_2,\ldots,Y_n]
  \]
  as $A$-algebras for some $n \ge 1$, do we have $A[X] \cong B$ as $A$-algebras?
\end{citedques}
Following \cite[p.\ 2]{Ham75},
we say that a ring $A$ is \textsl{steadfast} if Question \ref{ques:hamann} holds
for every ring extension $A \subseteq B$.
The aforementioned result due to Abhyankar, Eakin, and Heinzer
\cite[(4.5)]{AEH72} says that unique factorizations
domains are steadfast.
See \cite{EH73}, \cite[\S9]{Hoc75}, and \cite{Gup15}
for surveys on these and related
questions.\medskip
\par Not all rings are steadfast \citeleft\citen{Asa74}\citemid Corollary
3.3\citepunct \citen{Ham75}\citemid Example on p.\ 14\citeright.
However, Hamann showed that all rings containing a copy of the rational numbers
$\Q$ are steadfast \cite[Theorem 2.8]{Ham75}.
Moreover, Hamann \cite[Theorem 2.6]{Ham75}, 
Asanuma \cite[Theorem 3.13]{Asa78}, and Swan \cite[Theorem 9.1]{Swa80}
(in the reduced case) and 
Asanuma \cite[Proposition 1.6]{Asa82} and Hamann \cite[Theorem 2.4]{Ham83} (in
general)
showed that $A$ is steadfast if and only if
$A_\red \coloneqq A/N(A)$ is \textsl{$p$-seminormal} for all primes $p > 0$.
This notion of $p$-seminormality was introduced by Swan \cite[Definition on p.\
219]{Swa80} (cf.\ Asanuma's notion of an
\textsl{$F$-ring} \cite[Definition 1.4]{Asa78}), and was suggested by
Hamann's examples in \cite{Ham75}.
$0$-seminormality is equivalent to Traverso's notion of seminormality \cite[p.\
586]{Tra70}, which in turn is a generalization of the classical notion of
normality.
See Figure \ref{fig:variantsnormal} and \cite[(4.4)]{Kol16}
for the relationship between different
variants of normality, and see \cite{Vit11} for a survey.\medskip
\par In \cite[Main Theorem]{Hei08}, Heitmann showed that $0$-seminormality
deforms.
Here, we say that a property $\mathcal{P}$ of Noetherian local rings \textsl{deforms} if
the following condition holds: If $(A,\fm)$ is a Noetherian local ring and $y
\in \fm$ is a nonzerodivisor such that $A/yA$ is $\mathcal{P}$, then $A$ is $\mathcal{P}$.
This condition is called \textsl{lifting from Cartier divisors} in
\cite[Conditions 3.1]{Mur}, and is related to inversion of adjunction-type
results in birational geometry.
Many properties of Noetherian local rings such as regularity, normality,
and reducedness deform; see \cite[Table 2]{Mur} for a list of some known
results.
\par On the other hand, it has been an open question whether
$p$-seminormality deforms for $p \ne 0$.
By Asanuma and Hamann's results mentioned above, a closely related question is
whether steadfastness deforms.
Affirmative answers to these questions would give many ways to construct new
examples of steadfast rings.\medskip
\par In this paper, we answer both questions in the affirmative for reduced
Noetherian local rings.
\begin{customthm}{\ref{thm:pseminormallifts}}
  Let $(A,\fm)$ be a Noetherian local ring, and let $y \in \fm$ be a
  nonzerodivisor.
  \begin{enumerate}[label=$(\roman*)$,ref=\roman*]
    \item
      Let $p$ be an integer.
      If $A/yA$ is $p$-seminormal, then $A$ is $p$-seminormal.
    \item
      If $A/yA$ is reduced and steadfast, then $A$ is reduced and steadfast.
  \end{enumerate}
\end{customthm}
Theorem \ref{thm:pseminormallifts}$(\ref{thm:pseminormalliftsfixedp})$
extends Heitmann's theorem \cite[Main
Theorem]{Hei08}, which is the case when $p = 0$.
For some specific deformations, we can remove the hypothesis that $A$ is
Noetherian and local as follows.
\begin{customthm}{\ref{thm:pseminormalpowerseries}}
  Let $A$ be a ring, and let $X_1,X_2,\ldots$ be indeterminates.
  \begin{enumerate}[label=$(\roman*)$,ref=\roman*]
    \item
      Let $p$ be an integer. 
      If $A$ is $p$-seminormal, then $A\llbracket X_1,X_2,\ldots,X_n
      \rrbracket$ is $p$-seminormal for all $n \ge 0$.
    \item
      Suppose that $A$ is either reduced or Noetherian.
      If $A$ is steadfast, then $A\llbracket X_1,X_2,\ldots,X_n
      \rrbracket$ is steadfast for all $n \ge 0$.
  \end{enumerate}
\end{customthm}
Theorem \ref{thm:pseminormalpowerseries}$(\ref{thm:steadfastpowerseries})$
shows that even though steadfastness is
defined in terms of adjoining polynomial variables, it is well-behaved under
adjoining formal power series variables.
In Theorem \ref{thm:hs86forsteadfast}, we adapt an example of Hamann and Swan
\cite[Thm.\ 2.1]{HS86} to show that Theorem
\ref{thm:pseminormalpowerseries}$(\ref{thm:steadfastpowerseries})$ does not hold
for rings that are neither Noetherian nor reduced.
Theorem
\ref{thm:pseminormalpowerseries}$(\ref{thm:pseminormalpowerseriesfixedp})$
extends \cite[Theorem]{BN83}, which is the case when $p = 0$.
Note that the corresponding results for adjoining polynomial variables follow
from \cite[Theorem 6.1]{Swa80}.\medskip
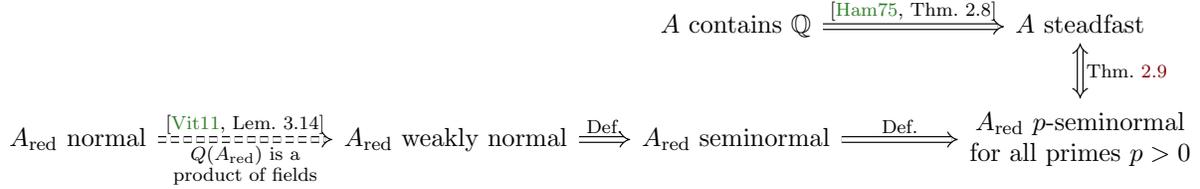
\begin{figure}[t]
    \begin{tikzcd}[scale cd=0.9,column sep=scriptsize]
      &[4.25em] & \text{$A$ contains $\Q$} \rar[Rightarrow]{\text{\cite[Thm.\ 
      2.8]{Ham75}}}
      &[2.25em] \text{$A$ steadfast} \dar[Leftrightarrow]{\text{Thm.\
      \ref{thm:steadfastiff}}}\\
      \begin{tabular}{@{}c@{}}
        $A_\red$ normal
      \end{tabular} \rar[dashed,Rightarrow]{\text{\cite[Lem.\
      3.14]{Vit11}}}[swap]{\substack{\text{$Q(A_\red)$ is a}\\
      \text{product of fields}}}
      & \text{$A_\red$ weakly normal} \rar[Rightarrow]{\text{Def.}}
      & \text{$A_\red$ seminormal}\rar[Rightarrow]{\text{Def.}}
      & \begin{tabular}{@{}c@{}}
        $A_\red$ $p$-seminormal\\
        for all primes $p > 0$
      \end{tabular}
    \end{tikzcd}
  \caption{Variants of normality and steadfastness.}
  \label{fig:variantsnormal}
\end{figure}
\par Finally, we use a
strategy similar to that for Theorem \ref{thm:pseminormallifts} to give new
proofs of the facts that normality and weak normality deform, which are of independent
interest.
Weak normality is a variant of normality in between normality and seminormality
introduced by Andreotti and Norguet for complex analytic spaces
\cite[Th\'eor\`eme 1, Remarque 1]{AN67}, and by Andreotti and Bombieri for
schemes \cite[Proposizione-Definizione 5]{AB69}.
\begin{customthm}{\ref{thm:nwnlift}}
  Let $(A,\fm)$ be a Noetherian local ring, and let $y \in \fm$ be a
  nonzerodivisor.
  If $A/yA$ is normal (resp.\ weakly normal), then $A$ is normal (resp.\ weakly
  normal).
\end{customthm}
The result that normality deforms is due to Seydi in this generality
\citeleft\citen{Sey72}\citemid Proposition I.7.4\citepunct \citen{Sey96}\citemid
Corollaire 2.8\citeright.
See \cite[Remark 1.4]{Chi82} and \cite[Lemma 0$(ii)$]{BR82} for alternative
proofs.
The result that weak normality deforms is due to Bingener and Flenner in the
excellent case \cite[Corollary 4.1]{BF93}, and to Murayama in general
\cite[Proposition 4.10]{Mur}.
The weakly normal case of Theorem \ref{thm:nwnlift} relies on \cite[Main
Theorem]{Hei08} (or Theorem \ref{thm:pseminormallifts}).
\subsection*{Outline}
This paper is organized as follows:
First, in \S\ref{section:background}, we review the definitions of and some
preliminary results on
steadfastness, (weakly) subintegral extensions, and the variants of normality
that appear in this paper.
In \S\ref{section:powerseries},
we show that $p$-seminormality and steadfastness are
preserved under adjoining formal power
series variables (Theorem \ref{thm:pseminormalpowerseries}).
The rest of the paper is devoted to showing Theorems \ref{thm:pseminormallifts}
and \ref{thm:nwnlift}.
In \S\ref{section:preliminary-steps}, we prove preliminary steps used in both of
these theorems.
We then prove that $p$-seminormality deforms (Theorem \ref{thm:pseminormallifts})
in \S\ref{section:pseminormalitylifts} by adapting the strategy in \cite{Hei08}.
Finally, \S\ref{section:nwnlift} is devoted to our new proofs of the facts that
normality and weak normality deform (Theorem \ref{thm:nwnlift}).

\subsection*{Conventions}
All rings are commutative with identity, and all ring homomorphisms are unital.
If $A$ is a ring, then $Q(A)$ denotes the total quotient ring of $A$,
$N(A)$ denotes the nilradical of $A$, and $A_\red$ denotes $A/N(A)$.

\subsection*{Acknowledgments}
This research was initiated as part of the
\href{https://sites.google.com/umich.edu/mreg2021}{2021 Michigan Research
Experience for Graduates}, which was (partially) supported by the
Rackham Graduate School through its
\href{https://rackham.umich.edu/faculty-and-staff/faculty-and-program-funding/rackham-faculty-allies-and-student-ally-diversity-grants/\#rfadg}{Faculty Allies for Diversity Program}.
We would like to thank the organizers and the Rackham Graduate School at the
University of Michigan for their support.
We are also grateful to
Rankeya Datta,
Neil Epstein,
J\'anos Koll\'ar,
Karl Schwede,
Kazuma Shimomoto,
Austyn Simpson,
and
Farrah Yhee
for helpful discussions.

\section{Preliminary background}\label{section:background}
In this section, we define and state preliminary results about steadfastness,
(weakly) subintegral extensions, and the variants of normality that appear in
this paper.
The only new result is a $p$-seminormal version of Hamann's criterion for
seminormality (Proposition \ref{prop:hamannforpseminormal}).
\subsection{Steadfastness}
We start by defining steadfastness.
\begin{citeddef}[{\cite[p.\ 2]{Ham75}}]
  Let $A$ be a ring, and let $X,X_1,X_2,\ldots$ and
  $Y_1,Y_2,\ldots$ be indeterminates.
  We say that $A$ is \textsl{steadfast} if, for every ring extension $A
  \subseteq B$ such that
  \[
    A[X,X_1,X_2,\ldots,X_{n}] \cong B[Y_1,Y_2,\ldots,Y_n]
  \]
  as $A$-algebras for some $n \ge 1$, we have $A[X] \cong B$ as $A$-algebras.
\end{citeddef}
\subsection{Subintegral and weakly subintegral extensions}
We now define seminormality and weak normality for extensions $A
\subseteq B$.
\begin{citeddef}[{\citeleft\citen{Swa80}\citemid p.\ 212\citepunct
  \citen{Yan85}\citemid p.\ 89 and p.\ 91\citeright}]
  Let $A \subseteq B$ be a ring extension.
  \begin{enumerate}[label=$(\roman*)$,ref=\roman*]
    \item We say that $A \subseteq B$ is a \textsl{subintegral extension}
      (resp.\ a \textsl{weakly subintegral extension}) if $B$
      is integral over $A$, $\Spec(B) \to \Spec(A)$ is a bijection, and
      $\kappa(\fq \cap A) \to \kappa(\fq)$ is an isomorphism (resp.\ a purely
      inseparable extension) for all $\fq \in \Spec(B)$.
    \item The \textsl{seminormalization} $\prescript{+}{B}{A}$ (resp.\ the
      \textsl{weak normalization} $\prescript{*}{B}{A}$) of $A$ in $B$
      is the largest
      subring of $B$ that is subintegral (resp.\ weakly subintegral) over $A$.
      Both $\prescript{+}{B}{A}$ and $\prescript{*}{B}{A}$ exist by
      \cite[Lemma 2.2]{Swa80} and \cite[Lemma 2]{Yan85}, respectively.
    \item We say that $A$ is \textsl{seminormal in $B$} (resp.\
      \textsl{weakly normal in $B$}) if $A = \prescript{+}{B}{A}$ (resp.\ $A =
      \prescript{*}{B}{A}$).
  \end{enumerate}
\end{citeddef}
\subsection{Variants of normality}
We now define $p$-seminormality, seminormality, and weak normality.
We start with intrinsic definitions of seminormality and weak normality,
following Swan and Yanagihara, respectively.
Swan included reducedness as part of the definition of
seminormality, but the condition in $(\ref{def:sn})$ below automatically implies
that $A$ is reduced by \cite[pp.\ 408--409]{Cos82}.
\begin{citeddef}[{\citeleft\citen{Swa80}\citemid Definition on p.\ 210\citepunct
  \citen{Yan85}\citemid p.\ 94\citeright}]\label{def:intrinsic}
  Let $A$ be a ring.
  \begin{enumerate}[label=$(\roman*)$,ref=\roman*]
    \item\label{def:sn} We say that $A$ is \textsl{seminormal} if
      for all $b,c \in A$ satisfying $b^3=c^2$, there exists $a
      \in A$ such that $a^2=b$ and $a^3=c$.
    \item\label{def:wn}
      We say that $A$ is \textsl{weakly normal} if $A$ is seminormal and if
      for all $b,c,e \in A$ and all nonzerodivisors $d \in A$ such that $c^p =
      bd^p$ and $pc = de$ for some prime $p > 0$, there exists $a \in A$ such
      that $b = a^p$ and $e = pa$.
  \end{enumerate}
\end{citeddef}
For rings such that $Q(A)$ is a product of fields, both notions can be
characterized in terms of the total quotient ring.
This condition on $Q(A)$ holds, for example, when $A$ is reduced and has only
finitely many minimal primes \cite[Corollary 3.6]{Swa80}, and in particular when
$A$ is reduced and Noetherian.
\begin{citedthm}[{\citeleft\citen{Swa80}\citemid p.\ 216\citepunct
  \citen{Yan85}\citemid Remark 1\citeright}]\label{thm:snwnwhenqaprod}
  Let $A$ be a ring such that $Q(A)$ is a product of fields.
  Then, $A$ is seminormal (resp.\ weakly normal) if and only if
  $A$ is seminormal (resp.\ weakly normal) in $Q(A)$.
\end{citedthm}
\begin{proof}
  For seminormality, this statement is shown in \cite[p.\ 216]{Swa80}.
  \par We now consider the statement for weak normality.
  First, we note that in either case $A$ is reduced, since $A \subseteq Q(A)$.
  Thus, the statement follows from \cite[Corollary to Proposition 4]{Yan85}
  since $Q(A)$ is weakly normal by
  \citeleft\citen{Yan85}\citemid Corollary to Proposition 2\citepunct
  \citen{Vit11}\citemid Lemma 3.13\citeright.
\end{proof}
\par To define $p$-seminormality, we use the following fact.
\begin{citedthm}[{\cite[Theorem 4.1]{Swa80}}]\label{thm:swa41}
  Let $A$ be a reduced ring.
  Then, there is a subintegral extension $A \subseteq B$ with $B$ seminormal.
  Any such extension is universal for maps of $A$ to seminormal rings:
  If $C$ is seminormal and $\varphi\colon A \to C$ is a ring map, then $\varphi$
  has a unique extension $\psi\colon B \to C$.
  Furthermore, if $\varphi$ is injective, then $\psi$ is injective.
\end{citedthm}
\begin{citeddef}[{\cite[p.\ 218]{Swa80}}]
  Let $A$ be a reduced ring.
  The ring $B$ in Theorem \ref{thm:swa41} is the \textsl{seminormalization} of
  $A$.
\end{citeddef}
We can now define $p$-seminormality.
\begin{citeddef}[{\cite[pp.\ 219--220]{Swa80}}]
  Let $p$ be an integer.
  \begin{enumerate}[label=$(\roman*)$,ref=\roman*]
    \item If $A \subseteq B$ is a ring extension, we say that $A$ is
      \textsl{$p$-seminormal} in $B$ if for all
      $x \in B$ satisfying $x^2,x^3,px \in A$, we have $x \in A$.
    \item A ring $A$ is \textsl{$p$-seminormal} if it is reduced and $A$ is
      $p$-seminormal in the seminormalization $B$ of $A$.
    \item A ring extension $A \subseteq B$ is \textsl{$p$-subintegral} if it is
      a filtered union of subextensions built up by finite sequences of
      extensions of the form $C \subseteq C[x]$ where $x^2,x^3,px \in C$.
    \item Let $A \subseteq B$ be a ring extension.
      The \textsl{$p$-seminormalization} of $A$ in $B$ is the largest
      subring $\prescript{+p}{B}A$ of $B$ that is $p$-subintegral over $A$.
  \end{enumerate}
\end{citeddef}
\begin{remark}\label{rem:swanremarks}
  We state some facts about $p$-seminormality from \cite[p.\ 219]{Swa80}.
  \begin{enumerate}[label=$(\roman*)$,ref=\roman*]
    \item By \cite[Theorem 2.5]{Swa80} (or Proposition
      \ref{prop:hamanncriterion} below), $0$-seminormality is the same as
      seminormality.
      By \cite[Theorem 2.8]{Swa80}, $0$-subintegral extensions are the same as
      subintegral extensions.
    \item\label{rem:pcanbeprime}
      If $p \ne 0$, a ring is $p$-seminormal if and only if it is
      $q$-seminormal for all primes $q \mid p$.
      We may therefore restrict to the case when $p \ge 0$ is
      either zero or a prime number.
    \item\label{rem:pseminormalintrinsic}
      There is an intrinsic definition of $p$-seminormality: $A$ is
      $p$-seminormal if and only if it is reduced and if for all $b,c,d \in A$
      such that $b^3=c^2$, $d^2=p^2b$, $d^3=p^3c$, there exists $a \in A$ such
      that $a^2=b$, $a^3=c$.
  \end{enumerate}
  We also note that by $(\ref{rem:pcanbeprime})$ above, a ring $A$ with finitely
  many minimal primes is
  $p$-seminormal for every prime $p > 0$ if and only if it is a $F$-ring in the
  sense of \cite[Definition 1.4]{Asa78}.
  See also \cite[Remark 2$(ii)$]{Yan93}.
\end{remark}
We can now state Asanuma and Hamann's characterization of steadfastness.
The case when $A$ is reduced is due to Hamann \cite[Theorem 2.6]{Ham75}, 
Asanuma \cite[Theorem 3.13]{Asa78}, and Swan \cite[Theorem 9.1]{Swa80}.
\begin{citedthm}[{\citeleft\citen{Asa82}\citemid Proposition 1.6\citepunct
  \citen{Ham83}\citemid Theorem 2.4\citeright}]\label{thm:steadfastiff}
  Let $A$ be a ring.
  The following are equivalent:
  \begin{enumerate}[label=$(\roman*)$,ref=\roman*]
    \item $A$ is steadfast.
    \item $A_\red$ is steadfast.
    \item $A_\red$ is $p$-seminormal for all primes $p > 0$.
  \end{enumerate}
\end{citedthm}
\subsection{Hamann's criterion}
We now state a version of Hamann's criterion for seminormality and the
corresponding result for weak normality.
The statement $(\ref{prop:hamanncriterionsn})$ is due to Hamann for Nagata
rings \cite[Propositions 2.10* and 2.11*]{Ham75}, and is proved in
\citeleft\citen{Rus80}\citemid Theorem 1\citepunct
\citen{LV81}\citemid Proposition 1.4\citeright\ in general; see also
\citeleft\citen{BC79}\citemid Theorem 1\citepunct
\citen{GH80}\citemid Theorem 1.1\citepunct
\citen{Swa80}\citemid Theorem 2.5\citeright.
The statement $(\ref{prop:hamanncriterionwn})$ is due to Itoh
\cite[Proposition 1]{Ito83} and Yanagihara \cite[Theorem 1]{Yan83}.
\begin{proposition}
  \label{prop:hamanncriterion}
  Let $A \subseteq B$ be an integral ring extension.
  \begin{enumerate}[label=$(\roman*)$,ref=\roman*]
    \item\label{prop:hamanncriterionsn}
      The following are equivalent:
      \begin{enumerate}[label=$(\alph*)$,ref=\alph*]
        \item $A \subseteq B$ is seminormal.
        \item For every element $x \in B$ satisfying $x^m,x^{m+1} \in A$ for some
          integer $m \ge 1$, we have $x \in A$.
      \end{enumerate}
    \item\label{prop:hamanncriterionwn}
      The following are equivalent:
      \begin{enumerate}[label=$(\alph*)$,ref=\alph*]
        \item $A \subseteq B$ is weakly normal. 
        \item $A \subseteq B$ is seminormal and for all $x \in B$ satisfying
          $x^p,px \in A$ for some prime $p > 0$, we have $x \in A$.
      \end{enumerate}
  \end{enumerate}
\end{proposition}
We prove that $p$-seminormality satisfies a version of Proposition
\ref{prop:hamanncriterion} as well.
\begin{proposition}\label{prop:hamannforpseminormal}
  Let $p$ be an integer, and consider an integral ring extension $A \subseteq
  B$.
  Then, $A$ is $p$-seminormal in $B$ if and only if
  for all $x \in B$ satisfying $x^m,x^{m+1},px \in A$ for some integer $m
  \ge 1$, we have $x \in A$.
\end{proposition}
\begin{proof}
  The direction $\Leftarrow$ follows by setting $m = 2$.
  \par To show $\Rightarrow$, we follow \cite[p.\ 282]{BN83}.
  Let $n \ge 1$ be the minimal integer such that 
  \[
    x^n,x^{n+1},x^{n+2},\ldots \in A.
  \]
  Such an $n$ exists since $n = m(m+1)-m-(m+1)+1 = m^2-m$ works by
  \cite[Theorem 2.1.1]{Alf05}.
Suppose that $n > 1$. In this case, we find $2(n-1)=2n - 2 \ge n$ and $3(n-1)=3n
-3 \ge n$. Thus,
\begin{align*}
  \Set[\big]{(x^{n-1})^2 = x^{2n-2},  (x^{n-1})^3 = x^{3n-3}, px} \subseteq A.
\end{align*}
By the assumption that $A$ is $p$-seminormal in $B$, we see that $x^{n-1} \in A$.
But this contradicts the minimality of $n$.
Hence, we see that $n = 1$, i.e., $x \in A$.
\end{proof}

\section{\texorpdfstring{$p$}{p}-seminormality in power series
rings}\label{section:powerseries}
In this section, we prove Theorem \ref{thm:pseminormalpowerseries}.
In Theorem \ref{thm:hs86forsteadfast}, we also show that one cannot remove the
``reduced or Noetherian'' assumption in Theorem
\ref{thm:pseminormalpowerseries}$(\ref{thm:steadfastpowerseries})$.\medskip
\par We start with the following lemma.
\begin{lemma}[cf.\ {\citeleft\citen{Swa80}\citemid Corollary 3.4\citepunct
  \citen{Yan93}\citemid Lemma 4\citeright}]
\label{lem:swa34}
Let $p$ be an integer.
Suppose we have a ring extension $A \subseteq B$ where $B$ is $p$-seminormal.
Then, $A$ is $p$-seminormal if and only if $A$ is $p$-seminormal in $B$.
\end{lemma}

\begin{proof}
Note that $A$ is reduced since it is a subring of $B$.
Let $A'$ be the seminormalization of $A$, and let $B'$ be the seminormalization
of $B$.
We have a commutative diagram
\[
  \begin{tikzcd}
    A \rar[hook]\dar[hook] & B\dar[hook]\\
    A' \rar[hook] & B'
  \end{tikzcd}
\]
with injective arrows by Theorem \ref{thm:swa41}.
\par For $\Rightarrow$,
consider an element $x \in B \subseteq B'$ such that $x^2,x^3,px \in A$.
Since $x^2,x^3 \in A$, we see that $x \in A'$ by \cite[Corollary 3.2]{Swa80}.
Since $A$ is $p$-seminormal, we then have $x \in A$.
\par For $\Leftarrow$, consider an element $x \in A'$ such that $x^2,x^3,px \in
A$.
Since $B$ is $p$-seminormal and $x \in A' \subseteq B'$ is an element such that
$x^2,x^3,px \in A \subseteq B$, we have $x \in B$.
But $A$ is $p$-seminormal in $B$. Hence $x \in A$.
\end{proof}

We now prove a stronger version of Theorem \ref{thm:pseminormalpowerseries} by
generalizing the strategy in \cite{BN83} used for $0$-seminormality.

\begin{theorem}[{cf.\ \cite[Theorem]{BN83}}]\label{thm:bn83}
  Let $p$ be an integer, and
  let $X_1,X_2,\ldots$ be indeterminates.
  \begin{enumerate}[label=$(\roman*)$,ref=\roman*]
    \item\label{thm:bn83i}
      Let $A \subseteq B$ be a ring extension.
      If $A$ is $p$-seminormal in $B$, then $A\llbracket
      X_1,X_2,\ldots,X_n\rrbracket$ is $p$-seminormal in $B\llbracket
      X_1,X_2,\ldots,X_n\rrbracket$ for all $n \ge 0$.
    \item\label{thm:bn83ii} Let $A$ be a ring.
      Then, $A$ is $p$-seminormal if and only if $A\llbracket
      X_1,X_2,\ldots,X_n\rrbracket$ is $p$-seminormal for all $n \ge 0$.
    \item\label{thm:bn83iimoreover} Let $A$ be a ring.
      If $A$ is $p$-seminormal and
      has only finitely many minimal primes, then $A\llbracket
      X_1,X_2,\ldots,X_n\rrbracket$ is $p$-seminormal in $Q(A\llbracket
      X_1,X_2,\ldots,X_n\rrbracket)$ for all $n \ge 0$.
  \end{enumerate}
\end{theorem}

\begin{proof}
  We first show $(\ref{thm:bn83i})$.
  By induction, it is enough to show the case when $n=1$, i.e., given $A$
  $p$-seminormal in $B$, we want to show that $A\llbracket X \rrbracket$ is
  $p$-seminormal in $B\llbracket X \rrbracket$.

      Let $f = \sum_{i=0}^\infty a_iX^i \in B\llbracket X \rrbracket$ such that
    $f^2, f^3, pf \in A\llbracket X \rrbracket$.
    We want to show $f \in A\llbracket X\rrbracket$.
    In fact, since $a_0 \in A$ by looking at constant terms and using the
    $p$-seminormality of $A$ in $B$, it suffices to show $(f - a_0)^2,(f - a_0)^3,p(f
    - a_0) \in A\llbracket X \rrbracket$, as we can replace $f$ with $(f-a_0)/X$ and iterate as needed.
      On expanding, we have
      \begin{alignat*}{3}
        p(f - a_0) &={}& pf &- pa_0,\\
        (f - a_0)^2 &={}& f^2 &- 2a_0f + a_0^2,\\
        (f - a_0)^3 &={}& f^3 &- 3a_0f^2 + 3 a_0^2f - a_0^3.
      \end{alignat*}
      Thus, it is enough to show that
      $a_0a_m \in A$ for all $m \in \Z_{\ge0}$.
      By induction on $m \in \Z_{\ge0}$, we will show that for all
      $\{n_1,n_2,\ldots,n_m\} \subseteq \Z_{\ge0}$, we have
      \[
        a_0a_1^{n_1}a_2^{n_2} \cdots a_m^{n_m} \in A \qquad \text{and} \qquad
        pa_1^{n_1}a_2^{n_2} \cdots a_m^{n_m} \in A.
      \]
      For $m = 0$, we have $a_0 \in A$ by the $p$-seminormality of $A$ in $B$
      as above, and we always have $p \in A$.
      \par For the inductive step, suppose the statement is true for $m$.
      Let
      \[
        \Tilde{f} = a_0 + a_1X + \cdots + a_mX^m.
      \]
      Let $n \geq 0$, and define
      \begin{align*}
        g &= (f - \Tilde{f})^{n} = \sum_{s = 0}^{n}
        \binom{n}{s} (-1)^s f^{n-s}\Tilde{f}^s.
      \intertext{Then, we have}
        a_0f^2g &= \sum_{s = 0}^{n}
        \binom{n}{s} (-1)^s f^{n-s+2}a_0\Tilde{f}^s \in A\llbracket X
        \rrbracket,
      \intertext{since $n-s+2 \geq 2$ implies $f^{n-s+2} \in A\llbracket
        X\rrbracket$ for all $s$, and
  $a_0\Tilde{f}^s \in A[X]$ by the inductive hypothesis.
  In particular, we find that the coefficient
  $a_0^3a_{m+1}^{n}$ of $X^{(m+1)n}$ in $f^2g$ lies in $A$.
      Similarly, we have}
        pa_0g &= \sum_{s = 0}^{n}
        \binom{n}{s} (-1)^s pf^{n-s}a_0\Tilde{f}^s \in A\llbracket X
        \rrbracket,\\
        p^2g &= \sum_{s = 0}^{n}
        \binom{n}{s} (-1)^s pf^{n-s}p\Tilde{f}^s \in A\llbracket X
        \rrbracket,
      \end{align*}
      since $pf^{n-s} \in A\llbracket X \rrbracket$ for all $s$ by hypothesis,
      and $a_0\tilde{f}^s,p\Tilde{f}^s \in A[X]$ by the inductive hypothesis.
      In particular, we find that the coefficients $pa_0a_{m+1}^{n}$ and
      $p^2a_{m+1}^n$ of $X^{(m+1)n}$ in $pa_0g$ and $p^2g$ respectively lie in
      $A$.
      \par We claim that $a_0a_{m+1}^n,pa_{m+1}^n \in A$ for all $n \ge
      0$.
      We have
      \begin{align*}
        (a_0a_{m+1}^n)^s &= a_0^{s-3}a_0^3a_{m+1}^{sn} \in A
        \intertext{for all $s \ge 3$ and $pa_0a_{m+1}^n \in A$ by the previous paragraph,
        and hence $a_0a_{m+1}^n \in A$ by the $p$-seminormality of $A$ in $B$
        and by Proposition \ref{prop:hamannforpseminormal}.
        Similarly, we have}
        (pa_{m+1}^n)^s &= p^{s-2}p^2a_{m+1}^{sn} \in A
      \end{align*}
      for all $s \ge 2$ and $p(pa_{m+1}^n) = p^2a_{m+1}^n \in A$
      by the previous paragraph, and hence $pa_{m+1}^n \in A$ for all $n \ge 0$
      by the $p$-seminormality of $A$ in $B$.
      \par Now, let $n_1, n_2, \ldots, n_{m+1} \in \mathbb{Z}_{\geq 0}$, and let
    $\alpha = a_0 a_1^{n_1}a_2^{n_2} \cdots a_{m + 1}^{n_{m + 1}}$ and $\beta =
    pa_1^{n_1}a_2^{n_2} \cdots a_{m + 1}^{n_{m + 1}}$.
      Then, we find
      \begin{align*}
        \alpha^s &= a_0^{s-4}(a_0a_1^{sn_1}a_2^{sn_2} \cdots
        a_m^{sn_m})a_0^3a_{m+1}^{sn_{m+1}} \in A
        \intertext{for all $s \ge 4$ by the inductive hypothesis, and}
        \beta^s &= p^{s-3}(pa_1^{sn_1}a_2^{sn_2} \cdots a_m^{sn_m})p^2a_{m+1}^{sn_{m+1}} \in A
      \end{align*}
      for all $s \geq 3$ by the inductive hypothesis.
      Moreover, we have
      \begin{align*}
        p\alpha &= (a_0a_1^{n_1}a_2^{n_2} \cdots a_m^{n_m})(pa_{m+1}^{n_{m+1}}) \in A\\
        p\beta &= (pa_1^{n_1}a_2^{n_2} \cdots a_m^{n_m})(pa_{m+1}^{n_{m+1}}) \in A
      \end{align*}
      for all $s \ge 2$ by the inductive hypothesis and the previous paragraph.
      Thus, by the $p$-seminormality of $A$ in $B$ and by
      Proposition \ref{prop:hamannforpseminormal}, we have
      $\alpha,\beta \in A$.
      Since $a_0a_m$ is of the form $\alpha$,
      this shows that $a_0a_m \in A$ for all $m \in \Z_{\ge0}$ as
      desired.\smallskip

      We now show $(\ref{thm:bn83ii})$.
  It is enough to show: If $A$ is $p$-seminormal, then $A\llbracket
  X_1,X_2,\ldots,X_n\rrbracket$ is $p$-seminormal.
      Let
      \[
        M = \Min(A) \coloneqq \Set[\big]{\mathfrak{p}_i \in \Spec(A) \given
        \mathfrak{p}_i\ \text{is a minimal prime}}.
      \]
      Observe:
      \begin{equation*}
        A \subseteq \prod_{\mathfrak{p}_i \in M}\frac{A}{\mathfrak{p}_i}
        \subseteq \prod_{\mathfrak{p}_i \in
        M}Q\biggl(\frac{A}{\mathfrak{p}_i}\biggr) \eqqcolon B.
      \end{equation*}
      Here, the first injection holds since $A$ is reduced, and hence
      $N(A) = \bigcap_{\fp_i
      \in M} \fp_i = 0$.
      Each $Q(A/\mathfrak{p}_i)$ is $p$-seminormal, since it is regular and
      therefore normal (see Figure \ref{fig:variantsnormal}).
      By \cite[Proposition 5.4]{Swa80}, $B$ is $p$-seminormal.
      Then, by Lemma \ref{lem:swa34}, $A$ is $p$-seminormal in $B$.

      Now each $Q(A/\mathfrak{p}_i)\llbracket X_1,X_2,\ldots,X_n \rrbracket$ is
      $p$-seminormal (it is a formal power series ring over a field, which is
      regular and therefore normal; see Figure \ref{fig:variantsnormal}), and
      from $(\ref{thm:bn83i})$, we find
      that $A\llbracket X_1,X_2,\ldots,X_n \rrbracket$ is $p$-seminormal in
      $B\llbracket X_1,X_2,\ldots,X_n \rrbracket$. But 
      \begin{equation*}
        B\llbracket X_1,X_2,\ldots,X_n \rrbracket \cong \prod_{\mathfrak{p}_i
        \in M} \biggl(Q\biggl(\frac{A}{\mathfrak{p}_i}\biggr)
        \llbracket X_1,X_2,\ldots,X_n \rrbracket\biggr)
      \end{equation*}
      which is $p$-seminormal by \cite[Proposition 5.4]{Swa80}.
      Then by Lemma \ref{lem:swa34} again, $A\llbracket X_1,X_2,\ldots,X_n
      \rrbracket$ is $p$-seminormal.\smallskip

      It remains to show $(\ref{thm:bn83iimoreover})$.
      As seen above, $A\llbracket X_1,X_2,\ldots,X_n
      \rrbracket$ is $p$-seminormal. Via Lemma \ref{lem:swa34}, it suffices to
      show that $Q(A\llbracket X_1,X_2,\ldots,X_n \rrbracket)$ is $p$-seminormal.

      Observe that if $A\llbracket X_1,X_2,\ldots,X_n \rrbracket$ has only
      finitely many minimal primes, then by \cite[Corollary 3.6]{Swa80}, we have
      that $Q(A\llbracket X_1,X_2,\ldots,X_n \rrbracket)$ is a direct product
      of fields (which are $p$-seminormal as before).
      By applying
      \cite[Proposition 5.4]{Swa80}, we get that $Q(A\llbracket
      X_1,X_2,\ldots,X_n \rrbracket)$ is $p$-seminormal. Therefore, it suffices
      to note that $A\llbracket X_1,X_2,\ldots,X_n \rrbracket$ has only finitely
      many minimal primes as is shown in the proof of \cite[Theorem (2)]{BN83}.
\end{proof}
In the Noetherian case, we have the following result.
\begin{proposition}[cf.\ {\cite[Proposition]{BN83}}]\label{prop:bn83}
Let $p$ be an integer, and
let $X_1,X_2,\ldots$ be indeterminates.
Let $A$ be a ring for which there exists a non-negative integer $m$
such that for all $a\in N(A)$, we have $a^m = 0$.
If $A_\red$
is $p$-seminormal, then
$( A\llbracket X_1,X_2,\ldots,X_n \rrbracket )_\red$
is $p$-seminormal for all $n \ge 0$.
In particular, this is true when $A$ is Noetherian.
\end{proposition}

\begin{proof}
  By induction, it is enough to prove that if $A_\red$ is $p$-seminormal, then
  $(A\llbracket X \rrbracket)_\red$ is $p$-seminormal.

  By \cite[Theorem 14]{Bre81}, the condition on $A$ implies that
  $N(A\llbracket X \rrbracket) = N(A)\llbracket X \rrbracket$.
  Thus,
  \begin{equation*}
    \bigl(A\llbracket X \rrbracket\bigr)_\red =
     \frac{A\llbracket X \rrbracket}{N(A)\llbracket X \rrbracket}
    \cong \biggl(\frac{A}{N(A)}\biggr) \llbracket X \rrbracket = A_\red
    \llbracket X \rrbracket.
  \end{equation*}
  When $A_\red$ is $p$-seminormal, we find that
  $A_\red\llbracket X
  \rrbracket$ is $p$-seminormal by Theorem \ref{thm:bn83}$(\ref{thm:bn83ii})$.
\end{proof}
We can now prove Theorem \ref{thm:pseminormalpowerseries}.
\begin{customthm}{B}\label{thm:pseminormalpowerseries}
  Let $A$ be a ring, and let $X_1,X_2,\ldots$ be indeterminates.
  \begin{enumerate}[label=$(\roman*)$,ref=\roman*]
    \item\label{thm:pseminormalpowerseriesfixedp}
      Let $p$ be an integer.
      If $A$ is $p$-seminormal, then $A\llbracket X_1,X_2,\ldots,X_n
      \rrbracket$ is $p$-seminormal for all $n \ge 0$.
    \item\label{thm:steadfastpowerseries}
      Suppose that $A$ is either reduced or Noetherian.
      If $A$ is steadfast, then $A\llbracket X_1,X_2,\ldots,X_n
      \rrbracket$ is steadfast for all $n \ge 0$.
  \end{enumerate}
\end{customthm}
\begin{proof}
  $(\ref{thm:pseminormalpowerseriesfixedp})$ is Theorem
  \ref{thm:bn83}$(\ref{thm:bn83ii})$.
  For $(\ref{thm:steadfastpowerseries})$, one combines Asanuma and Hamann's
  characterization of steadfastness (Theorem \ref{thm:steadfastiff}) together
  with $(\ref{thm:pseminormalpowerseriesfixedp})$ in the reduced case or with
  Proposition \ref{prop:bn83} in the Noetherian case.
\end{proof}
\begin{remark}
  One can also ask whether the analogue of Theorem \ref{thm:bn83} holds
  for weak normality in the sense of Definition
  \ref{def:intrinsic}$(\ref{def:wn})$.
  Dobbs and Roitman proved this when $A$ is a domain \cite[Theorem 3 and
  Corollary 4]{DR95}, and
  Maloo proved the analogue of Theorem \ref{thm:bn83}$(\ref{thm:bn83i})$
  in general \cite[Theorem 9$(b)$]{Mal96}.
  Maloo's result implies the analogue of the special case of
  Theorem \ref{thm:bn83}$(\ref{thm:bn83ii})$ when $Q(A)$ is a product of fields
  by using the characterization in Theorem \ref{thm:snwnwhenqaprod}, and by
  replacing \cite[Proposition 5.4]{Swa80} and Lemma \ref{lem:swa34}
  with \citeleft\citen{Yan85}\citemid Corollary to Proposition 2\citepunct
  \citen{Vit11}\citemid Lemma 3.13\citeright\ and \cite[Propositions 3 and
  4]{Yan85}, respectively.
  This in turn yields the analogue of the special case of Proposition
  \ref{prop:bn83} when $Q(A)$ is a product of fields.
  This also yields the complete
  analogue of Theorem \ref{thm:bn83}$(\ref{thm:bn83iimoreover})$.
\end{remark}

In \cite[Question on p.\ 284]{BN83}, Brewer and Nichols asked: If a ring $A$ is such that $A_\red$ is seminormal (with no additional conditions such as Noetherianity, etc.), is it always true that $(A\llbracket x\rrbracket)_\red$ is also seminormal? This was answered in the negative by Hamann and Swan in \cite{HS86}, whose construction we modify to get the following result.
\begin{theorem}[cf.\ {\cite[Theorem 2.1]{HS86}}]\label{thm:hs86forsteadfast}
For every prime $q > 0$, there exists a non-Noetherian non-reduced ring $S$ of
characteristic $q$ such that $S_\red$ is $p$-seminormal for every integer $p > 0$, but
$(S\llbracket X\rrbracket)_\red$ is not $q$-seminormal.
In particular, there exists a ring $S$ such that
$S$ is steadfast, but $S\llbracket X \rrbracket$ is not steadfast.
\end{theorem}

By this result, one needs some finiteness conditions on a
ring for steadfastness to be inherited when adjoining a formal power series
variable.
In the Noetherian case, we can deduce the $p$-seminormality of $(A\llbracket X_1,X_2,\ldots,X_n \rrbracket)_\red$ from the $p$-seminormality of $A_\red$ by Proposition \ref{prop:bn83}.

\begin{proof}
  The ``in particular'' statement follows from Theorem \ref{thm:steadfastiff},
  and hence it suffices to show the statement about $p$-seminormality.

Let $K$ be a field, and consider the ring $R = K[\{b_\beta\},\{c_\gamma\},\{d_\delta\}]_{\beta, \gamma, \delta \in \Z_{\geq 0}} $ where $\{b_\beta\}$, $\{c_\gamma\}$, and $\{d_\delta\}$ are sets of independent indeterminates.
Consider the ideal $I$ generated by the following subset of $R$:
\[
  \Set[\bigg]{b_0^2, b_1^2, b_t^t, c_1^2, c_t^t, d_1^2, d_t^t, \sum_{i+j+k
  = n}b_ib_jb_k - \sum_{i+j = n}c_ic_j, \sum_{i+j= n}d_id_j - q^2 b_n, \sum_{i+j
  +k= n}d_id_jd_k - q^3 c_n \given t \geq 2}.
\]
Set $S\coloneqq R/I$, which is constructed so that $b_\beta,c_\gamma,d_\delta$ are nilpotent elements, and 
\begin{equation*}
  B(X)^3 - C(X)^2 = D(X)^2 - q^2B(X) = D(X)^3-q^3C(X) = 0,
\end{equation*}
where we use the generic notation $F(X) = \sum_{i=0}^\infty f_iX^i$.
We will show that $S$ is an example of the desired form when $K$ is of
characteristic $q$.

First, we note that $S_\red = K$, which is a field. Since fields are regular,
they are normal and therefore $p$-seminormal for all $p>0$ (see Figure \ref{fig:variantsnormal}).

It remains to show that $(S\llbracket X\rrbracket)_\red$ is not $q$-seminormal.
We proceed by contradiction.
Suppose $(S\llbracket X\rrbracket)_\red$ is $q$-seminormal.
By Remark \ref{rem:swanremarks}$(\ref{rem:pseminormalintrinsic})$,
there exists some $A(X)$ such that $B(X) = A(X)^2 + H(X)$ in $S\llbracket X \rrbracket$, where $H(X)$ is nilpotent with some finite nilpotence degree $N$. 
Choose an even integer $n > N/2$, and define the ideal
\[
  J = \bigl\langle \{b_\beta\},\{c_\gamma\},\{d_\delta\} \bigm\vert
  \beta \neq 2n,\gamma \ne 3n \bigr\rangle.
\]
Let $b = b_{2n}$ and $c = c_{3n}$. Since our field $K$ is of characteristic $q$, we find
\begin{equation*}
    \bar{S} \coloneqq \frac{S}{J} = \frac{K[b,c]}{\langle b^{2n}, c^{3n}, b^3 - c^2
    \rangle}= \frac{K[b,c]}{\langle b^{2n}, b^3 - c^2 \rangle},
\end{equation*}
where we use the definition of $J$ and the characteristic of $K$ to make the relevant reductions. So 
\begin{equation*}
    bX^{2n} = A(X)^2 + H(X)
\end{equation*}
in $\bar{S}\llbracket X\rrbracket$.
Note that $N(\bar{S}) = \langle b,c \rangle$, and since $\bar{S}$ is Noetherian,
we have $N(\bar{S}\llbracket X\rrbracket) = \langle b,c \rangle\bar{S}\llbracket
X\rrbracket$ by \cite[Theorem 14]{Bre81}.
The nilpotency of $H(X)$ tells us that $H(X) \in \langle b,c \rangle\bar{S}\llbracket X\rrbracket$.
Thus, since $A(X)^2 = 0$ in $K\llbracket X\rrbracket$, we have $A(X) = 0$ in
$K\llbracket X\rrbracket$ and $A(X) \in \langle b,c \rangle\bar{S}\llbracket X\rrbracket$.

Write $A(X) = bA_0(X) + cA_1(X)$. Then, $A(X)^2 \in b\langle
b,c\rangle\bar{S}\llbracket X\rrbracket$ by virtue of the relation $b^3 - c^2 =
0$ in $\bar{S}$, so $A(X)^2 = b E(X)$ with $E(X) \in \langle b,c
\rangle\bar{S}\llbracket X\rrbracket$.
Thus, $E(X)$ is nilpotent, and in $\bar{S}\llbracket X\rrbracket$, we have
\begin{equation*}
    bX^{2n} - bE(X) = H(X) \implies b^N\bigl(X^{2n}-E(X)\bigr)^N =
    \bigl(H(X)\bigr)^N = 0.
\end{equation*}
By \cite[Lemma 1.1]{HS86}, $X^{2n} - E(X)$ is a nonzerodivisor, and hence $b^N =
0$ in $\bar{S}$.
Since $2n > N$, we have a contradiction.
Therefore,
$(\bar{S}\llbracket X\rrbracket)_\red$ is not $q$-seminormal.
\end{proof}

\begin{remark}
  Theorem \ref{thm:hs86forsteadfast} shows that without some finiteness
  conditions, for each prime $q>0$ we can find a steadfast ring $A$ such that
  $(A\llbracket X \rrbracket)_\red$ is not $q$-seminormal.
  On the other hand, this construction would not work if one chooses $K$ to be of characteristic zero in the proof.
\end{remark}

\section{Preliminary steps for Theorems \ref{thm:pseminormallifts} and \ref{thm:nwnlift}}\label{section:preliminary-steps}
We will perform four preliminary steps in this section that will be used in the following sections to prove Theorems \ref{thm:pseminormallifts} and \ref{thm:nwnlift}.
These steps follow the strategy in \cite{Hei08} for 0-seminormality.\medskip
\par We will proceed by contradiction, for which
we fix the following notation.
\begin{setup}\label{settingforthmsaandc}
  Let $\mathcal{P}$ be one of the following properties: $p$-seminormal (for a
  fixed integer $p$), weakly normal, or normal.
Let $(A,\mathfrak{m})$ be a Noetherian local ring, and let $y \in \mathfrak{m}$
be a nonzerodivisor such that $A/yA$ satisfies $\mathcal{P}$.
Note that $A$ is reduced by \cite[Proposition 3.4.6]{EGAIV2} (see also
\cite[Lemma 0$(ii)$]{BR82}).
\par Suppose furthermore that $A$ does not satisfy $\mathcal{P}$.
We can choose an element $x \in Q(A) - A$ that witnesses this in the following
manner.
\begin{itemize}
  \item If $\mathcal{P} =$ ``$p$-seminormal,'' we can choose $x$ such that
    $x^2,x^3,px \in A$.
  \item If $\mathcal{P} =$ ``weakly normal,'' then $A$ is seminormal by
    \cite[Main Theorem]{Hei08} (or Theorem \ref{thm:pseminormallifts}), and
    hence we can choose $x$ such that $x^p,px \in A$ for some prime $p > 0$.
  \item If $\mathcal{P} =$ ``normal,'' we can choose $x$ such that $x$ is
    integral over $A$.
\end{itemize}
\end{setup}

We now prove a series of lemmas, which allow us to assume that our
counterexample $A$ has additional properties.

\begin{lemma}[cf.\ {\cite[Lemma 1]{Hei08}}]\label{step1}
  In the setting of Setup \ref{settingforthmsaandc},
  we may assume that our counterexample $(A,\fm)$ with $x \in Q(A) - A$ is such that
  $(A :_A x) \eqqcolon \mathfrak{q}$ is a prime ideal of $A$ and that $\mathfrak{m}$ is minimal over $\mathfrak{q} + y A$, with $\height (\mathfrak{m}) \geq 2$.
\end{lemma}

\begin{proof}
Since $x\not\in A$, we note that $(A:_A x)\subsetneq A$ is a proper ideal. By \cite[Theorem 7.13]{AM69}, we can then consider an irredundant primary decomposition
\begin{align*}
  (A:_A x) &= \mathfrak{q}_1\cap \mathfrak{q}_2\cap\cdots\cap \mathfrak{q}_n.   
\end{align*}
Define $\mathfrak{q} =\sqrt{\mathfrak{q}_1}$. By \cite[Proposition 7.17]{AM69}, there exists an element $r\in A$ such that
\begin{align*}
    \mathfrak{q}=\bigl((A:_A x) :_A r\bigr) = (A:_A rx),
\end{align*}
where the equality holds by the equivalences
\begin{align*}
    s \in (A :_A rx) &\Longleftrightarrow srx \in A \\
    &\Longleftrightarrow sr \in (A :_A x) \\
    &\Longleftrightarrow s \in \bigl((A :_A x) :_A r\bigr).
\end{align*}
\par Now let $\mathfrak{p}$ be a minimal prime of $\mathfrak{q} + yA$. Since
$\mathcal{P}$ is stable under localization \citeleft\citen{Swa80}\citemid
Proposition 5.4\citepunct \citen{Yan83}\citemid Corollary 1\citepunct
\citen{AM69}\citemid Proposition 5.13\citeright,
we know that $(A/yA)_\fp$ satisfies $\mathcal{P}$.
On the other hand, $rx \notin A_\fp$ since
	\[
	 (A_{\mathfrak{p}}:_{A_{\mathfrak{p}}}rx) =(A:_A rx)A_{\mathfrak{p}} = \mathfrak{q} A_{\mathfrak{p}}\subsetneq A_{\mathfrak{p}}.
	\]
  Note that for $p$-seminormality, we still have $(rx)^2, (rx)^3,prx \in
  A_{\mathfrak{p}}$, for weak normality, we still have
  $(rx)^p, prx \in A_{\mathfrak{p}}$, and for normality, we still have that $rx$
  is integral over $A_\mathfrak{p}$.

  We claim we may replace $A$ by $A_\fp$ and $x$ by $rx$.
  Set $\mathfrak{m} = \fp A_\fp$.
  We are left to show that $\height (\mathfrak{m}) \geq 2$. Suppose by way of contradiction that $\height (\mathfrak{m}) < 2$. Then $\dim(A/yA) = 0$. Since $A/yA$ is reduced, it is regular. We therefore see that $A$ is regular by \cite[Chapitre 0, Corollaire 17.1.8]{EGAIV2}. So $A$ is normal, and hence satisfies $\mathcal{P}$ (see Figure \ref{fig:variantsnormal}). This is a contradiction.
\end{proof}

The next step will only be used in the proof of Theorem
\ref{thm:pseminormallifts} (when $\mathcal{P} =$ ``$p$-seminormal''), but we
prove it for all three cases.
\begin{lemma}[cf.\ {\cite[Lemma 2]{Hei08}}]\label{step2}
  In the setting of Setup \ref{settingforthmsaandc} and Lemma \ref{step1},
  we may further assume that our counterexample $(A,\fm)$ 
  is $y$-adically complete.
\end{lemma}

\begin{proof}
Let $S$ be the $y$-adic completion of $A$, i.e.,
\[S = \varprojlim_n (A/y^n A).\]
The quotient ring $S/yS$ satisfies $\mathcal{P}$ since $S/yS \cong A/yA$. The fact that $S$ is local follows from the fact that maximal ideals in $S/yS$ correspond to maximal ideals in $S$ containing $y$, and by \cite[Proposition 10.15$(iv)$]{AM69} every maximal ideal of $S$ contains $y$. We have $x^2, x^3 , px\in A \subseteq S$, but $x\not\in S$, since otherwise by
\cite[Chapter I, \S2, n\textsuperscript{o} 10, Proposition 12 and Remark on p.\
24]{Bou72}, we have
\[(A:_A x)S = (S :_S x) = S,\]
which cannot occur by \cite[Theorem 7.5$(ii)$]{Mat89} since $A \to S$ is
faithfully flat \cite[Theorem 8.14(3)]{Mat89}.
Here, we think of $x$ as an element of $Q(S)$, since nonzerodivisors in $A$ map
to nonzerodivisors in $S$, and hence $A \to S$ induces an inclusion $Q(A)
\hookrightarrow Q(S)$.

Finally, we must check that we may ensure that the conclusions of Lemma \ref{step1} can still be satisfied.

Let $\mathfrak{q}_1$ be a minimal prime of $(S:_S x)$. By \cite[Proposition 7.17]{AM69}, there exists $s \in S$ such that
\[\mathfrak{q}_1 = \bigl((S :_S x):_S s\bigr) = (S :_S sx).\]
  Note that for $p$-seminormality, we still have $(sx)^2, (sx)^3,psx \in
  S$, for weak normality, we still have
  $(sx)^p, psx \in S$, and for normality, we still have that $sx$
  is integral over $S$.

Denoting by $(A/\mathfrak{m})^\wedge$ the $y$-adic completion of
$A/\mathfrak{m}$, we have 
\[
  S/\mathfrak{m}S \cong A/\mathfrak{m} \otimes_A S \cong (A/\mathfrak{m})^\wedge
  \cong A/\mathfrak{m},
\]
where the middle isomorphism follows from \cite[Proposition 10.13]{AM69} and the
last isomorphism follows
since $y \in \mathfrak{m}$.
Thus, $\mathfrak{m}S$ is the maximal ideal of $S$ since $\mathfrak{m}$ is maximal in $A$. Next, we will show that $\mathfrak{m}S$ is a minimal prime of $\mathfrak{q}_1 + yS$.
If $\mathfrak{n}$ is a prime of $S$ with $\mathfrak{q}_1 + yS \subseteq
\mathfrak{n} \subseteq \mathfrak{m}S$, then when we contract to $A$, we get
\[\mathfrak{q} + yA \subseteq (\mathfrak{q}_1 + yS) \cap A \subseteq
\mathfrak{n} \cap A \subseteq \mathfrak{m}S \cap A = \mathfrak{m},\]
where the last equality follows from \cite[Proposition 10.15$(iv)$]{AM69}. By
the minimality of $\mathfrak{m}$ over $\mathfrak{q} + yA$, we must have
$\mathfrak{n} \cap A = \mathfrak{m}$, so $\mathfrak{m}S \subseteq \mathfrak{n}$. Since $\mathfrak{m}S$ is maximal, we have $\mathfrak{m}S = \mathfrak{n}$, so $\mathfrak{m}S$ is a minimal prime of $\mathfrak{q}_1 + yS$.

By virtue of the going down theorem for flat extensions \cite[Theorem 9.5]{Mat89}, we conclude that $\height(\mathfrak{m}S) \geq \height(\mathfrak{m}) \geq 2$. This completes the conclusions of the first step, allowing us to replace $A$ with its $y$-adic completion $S$, $x$ with $sx$, and $\mathfrak{q}$ with $\mathfrak{q}_1$.
\end{proof}
We now begin proving some preliminary steps towards Theorems
\ref{thm:pseminormallifts} and \ref{thm:nwnlift}.
\begin{lemma}[cf.\ {\cite[p.\ 441]{Hei08}}]\label{step3}
  In the setting of Setup \ref{settingforthmsaandc} and Lemma \ref{step1}, our
  counterexample $(A,\fm)$ satisfies the following conditions.
  \begin{enumerate}[label=$(\roman*)$,ref=\roman*]
    \item\label{step3btilde} Let $A'$ be the normalization of $A$. Then,
     $\tilde{B} \coloneqq \{t \in A' \mid \mathfrak{q}t \subseteq A\}$ is a subring of $A'$.
    \item\label{step3bmodfin} Every subring $B$ of $\tilde{B}$ such that $A \subseteq B$ is module-finite over $A$, and satisfies $\mathfrak{q}B = \mathfrak{q}$.
  \end{enumerate}
\end{lemma}
\begin{proof}

Since $\mathfrak{q} \cdot 1_{A'} = \mathfrak{q} \subseteq A$, we have $1_{A'} \in B$. So to show that $\tilde{B}$ is a subring of $A'$ it now suffices to show that $\tilde{B}$ is stable under multiplication and subtraction. To see this, let $t,t'\in \tilde{B}$ and $q\in \mathfrak{q}$. Then we have $qtt' = (qt)t'\in (\mathfrak{q} \tilde{B})\tilde{B}= \mathfrak{q} \tilde{B}\subseteq A$ and $q(t-t')=qt-qt'\in A$. This proves $(\ref{step3btilde})$. 

For $(\ref{step3bmodfin})$,
let $B$ be a subring of $\tilde{B}$ that contains $A$. We first show $\mathfrak{q}B = \mathfrak{q}$. Since $B$ contains $A$ and is a ring, we have $1_{A'} \in B$ and so $\mathfrak{q} \subseteq \mathfrak{q}B$. Conversely, we have $\mathfrak{q} B =  \mathfrak{q} B \cap A \subseteq \mathfrak{q} A' \cap A$; since $\mathfrak{q}$ is radical and $A'$ is integral over $A$, this implies by \cite[Remark 3]{McQ79} that $\mathfrak{q} A' \cap A = \mathfrak{q}$, so $\mathfrak{q}B \subseteq \mathfrak{q}$.

It remains to show that $B$ is module-finite over $A$. Since $x \in Q(A)$ is nonzero, there exists a nonzerodivisor $q \in A$ such that $qx \in A$, in which case $q \in \mathfrak{q}$ by Lemma \ref{step1}. Consider the composition of $A$-module homomorphisms
\[
  \begin{tikzcd}[column sep=scriptsize]
    B \arrow[r, "q\cdot -", "\sim"'] & qB \arrow[r, hook] &[-0.425em] \mathfrak{q}B=\mathfrak{q}
\end{tikzcd}
\]
where the first map is an isomorphism since $q$ is a nonzerodivisor. This composition realizes $B$ as an $A$-submodule of the finitely generated $A$-module $\mathfrak{q}$. Since $A$ is Noetherian, $B$ is module-finite over $A$. 
\end{proof}

\begin{lemma}[cf.\ {\cite[Lemma 3]{Hei08}}]\label{step4}
  Let $(A,\fm)$ be our counterexample as in
  Setup \ref{settingforthmsaandc} and Lemma \ref{step1},
  let $\tilde{B}$ be as in Lemma \ref{step3}$(\ref{step3btilde})$, and let $K = Q(A/yA)$.
  Let $B$ be one of the following subrings of $\tilde{B}$ containing $A$.
  \begin{enumerate}[label=$(\roman*)$,ref=\roman*]
    \item\label{step4:npsn}
      $B = \tilde{B}$.
    \item\label{step4:wn}
      $B = \Set{t \in \tilde{B} \given t^p \in A}$ (if $\mathcal{P} =$
      ``weakly normal'').
  \end{enumerate}
  Then, there exists a commutative diagram
\[
\begin{tikzcd}
A/yA \arrow[r, hook]                   & B/yB \arrow[r, hook]         & K \\
A \arrow[r, hook] \arrow[u, two heads] & B \arrow[u, two heads] &  
\end{tikzcd}\]
where the horizontal arrows are injective.
\end{lemma}

\begin{proof}
We first note that $B = \tilde{B}$ is a subring of $A'$ in
case $(\ref{step4:npsn})$ by Lemma \ref{step3}$(\ref{step3btilde})$.
We also show that in case $(\ref{step4:wn})$, the set $B$ is a subring of
$\tilde{B}$ containing $A$.
Since $A \subseteq B$, we have $1 \in B$.  Hence, it suffices to show that $B$
is stable under multiplication and subtraction. If $s, t \in B$, then 
$(st)^p = s^p t^p \in A$, and hence $B$ is stable under multiplication.
To show that $B$ is stable under subtraction, we first note that $(s - t)^p
= s^p - t^p + p b$ for some $b \in \tilde{B}$. Since $p \in \mathfrak{q}$ (by
the fact that $px \in A$ in Setup \ref{settingforthmsaandc}), we have $pb \in
\fq B = \fq \subseteq A$, so $(s - t)^p \in A$, and hence $s - t \in B$.
Thus, $B$ is stable under subtraction.

Next, we see that
the commutativity of the square follows from basic properties of quotient rings.
Thus, we need to show that the inclusion $A/yA \hookrightarrow K$ factors through $B/yB$, and that the maps on the top row are injective.

We will first construct a nonzerodivisor $c \in \mathfrak{q}$ whose image $\bar c \in A/yA$ is also a nonzerodivisor. If $\mathfrak{q}$ is contained in a minimal prime $\mathfrak{p}$ of $yA$, then $\mathfrak{q} \subseteq \mathfrak{q} + yA \subseteq \mathfrak{p}$.
Thus, by Lemma \ref{step1}, we have $\mathfrak{m} = \mathfrak{p}$, which implies that $\mathfrak{m}$ has height $1$, contradicting Lemma \ref{step1}. By \cite[Lemma 3.5]{Swa80},
the set of zero divisors of a reduced ring is the union of its minimal primes. As in the proof of Lemma \ref{step3}$(\ref{step3bmodfin})$, the denominator of $x \in Q(A)$ is a nonzerodivisor in $\mathfrak{q}$, so $\mathfrak{q}$ is not a minimal prime of $0 \subseteq A$. Let $\mathfrak{p}_1,\mathfrak{p}_2,  \ldots, \mathfrak{p}_n$ and $\mathfrak{p}'_1, \mathfrak{p}'_2,\ldots, \mathfrak{p}'_m$ be the minimal primes of $0$ and $yA$ respectively. We have $\mathfrak{q} \not\subseteq\mathfrak{p}_i$ and $\mathfrak{q} \not\subseteq\mathfrak{p}'_i$ for each $i$. Therefore by prime avoidance \cite[Proposition 1.11]{AM69}, we see that
\[\mathfrak{q} \not \subseteq \left(\bigcup_i \mathfrak{p}_i\right) \cup \left(\bigcup_i \mathfrak{p}'_i\right).\]
Take $c$ to be a witness to this noninclusion. Then $c$ is in none of the minimal primes in either $A$ or $A/yA$, so neither $c$ nor $\bar c$ is a zero divisor.

Define a map $B \rightarrow K$ by $b \mapsto (\overline{bc})/\bar c$. This is well defined since $bc \in \mathfrak{q}B = \mathfrak{q} \subseteq A$ (by Lemma \ref{step3}$(\ref{step3bmodfin})$) and since $\bar c$ is a nonzerodivisor on $A/yA$ by assumption. Under this map, note that $y \mapsto (\bar y \bar c)/\bar c = \bar y  = 0$, so it induces a map $B/yB \rightarrow K$. Note that the composition $A/yA \rightarrow B/yB \rightarrow K$ is induced by the map $A \rightarrow K$ given by $r \mapsto (\bar r \bar c)/\bar c = \bar r$, so it agrees with the inclusion $A/yA \hookrightarrow K$, and hence we have the existence of the top row of the diagram.

It remains to show that the homomorphisms $A/yA \to B/yB$ and $B/yB \to K$ are injective. Since the composition $A/yA \hookrightarrow K$ is injective, the homomorphism $A/yA \to B/yB$ is injective. So we have $yB \cap A = yA$. 

Now we show that $B/yB \to K$ is also injective. Pick any $b \in B$ such that its image $\overline{b}$ in $B/yB$ lies in $\ker(B/yB \to K)$. We want to show that $\overline{b} = 0$, or equivalently, $b \in yB$. Let $\overline{c}$ be the image of $c$ in $B/yB$. Then $\overline{c}\overline{b} \in \ker(B/yB \to K)$. Since $c \in \mathfrak{q}$ and $b \in B$, we have $cb \in A$. So we can lift $\overline{c}\overline{b}$ to $cb \in A$, and obtain $cb \in \ker(A \to K) = yA \subseteq yB$. This implies $\overline{c}\overline{b} = 0 \in B/yB$. Since $A/yA \to B/yB$ is injective and $c \not\in yA$, we know $\overline{c} \neq 0$. To show that $\overline{b} = 0$, it suffices to show that $\overline{c}$ is a nonzerodivisor on $B/yB$. This is equivalent to showing that if there is some $d \in B$ such that $cd \in yB$, then $d \in yB$. 

First we show that $cd^n \in y^n A$ for every integer $n > 0$. We proceed by induction on $n$. Consider the $n=1$ case. Since $c \in \mathfrak{q}$ and $d \in B$, we have $cd \in A$ by definition of $B$. Since $cd \in yB$ by our choice of $d$ and $yB \cap A = yA$, we have $cd \in yA$. Now suppose $n > 1$. By inductive hypothesis, we have
\[c \cdot cd^n = c^2d^n = cd \cdot cd^{n-1} \in (y \cdot y^{n-1})A = y^nA.\]
Because $B$ is a ring, we have $cd^n \in A$ by definition of $B$.
Since $y,c$ is a regular sequence on $A$ by construction of $c$, the sequence $y^n,c$ is also a regular sequence on $A$ by
\cite[Theorem 16.1]{Mat89}, and hence $c$ is a nonzerodivisor on $A/y^nA$. Thus,
we have $cd^n \in y^n A$.
We therefore see that $c(d/y)^n \in A$ for all integers $n > 0$, which shows 
$d/y \in A'$ by \cite[p.\ 133 and Theorem 13.1(3)]{Gil92}.

We now show that $d/y \in \tilde{B}$.
By the fact that $cd \in yA$, we have $c\mathfrak{q}d \subseteq yA$.
Since $\mathfrak{q}d \subseteq A$ by definition of $\tilde{B}$ and $c$ is a
nonzerodivisor on $A/yA$, we know that $\mathfrak{q}d \subseteq yA$. Thus,
$\mathfrak{q} \cdot (d/y) \subseteq A$. By definition of $\tilde{B}$, we have
$d/y \in \tilde{B}$, and therefore $d \in y\tilde{B}$. This shows that
$\overline{c}$ is a nonzerodivisor on $\tilde{B}/y\tilde{B}$, concluding the
proof in case $(\ref{step4:npsn})$.

It remains to consider case $(\ref{step4:wn})$.
We need to show that $d/y \in B$.
By the argument above, we have $cd^p \in y^pA$.
As before, we know that $y^p,c$ is a regular sequence on $A$ by \cite[Theorem
16.1]{Mat89}, and hence $d^p \in y^pA$, which shows 
$(d/y)^p \in A$.
By definition of $B$, we have $d/y \in B$, and therefore $d \in yB$.
This shows that $\overline{c}$ is a nonzerodivisor on $B/yB$, concluding the
proof in case $(\ref{step4:wn})$.
\end{proof}

\section{\texorpdfstring{$p$}{p}-seminormality
deforms}\label{section:pseminormalitylifts}
In this section, we show that $p$-seminormality deforms.
\begin{customthm}{A}\label{thm:pseminormallifts}
  Let $(A,\fm)$ be a Noetherian local ring, and let $y \in \fm$ be a
  nonzerodivisor.
  \begin{enumerate}[label=$(\roman*)$,ref=\roman*]
    \item\label{thm:pseminormalliftsfixedp}
      Let $p$ be an integer.
      If $A/yA$ is $p$-seminormal, then $A$ is $p$-seminormal.
    \item\label{thm:steadfastlifts}
      If $A/yA$ is reduced and steadfast, then $A$ is reduced and steadfast.
  \end{enumerate}
\end{customthm}
By Asanuma and Hamann's characterization of steadfastness (Theorem
\ref{thm:steadfastiff}), it suffices to show
$(\ref{thm:pseminormalliftsfixedp})$.
We proceed by contradiction. To do so, we use the setup from Setup
\ref{settingforthmsaandc} and
the preliminary steps from \S\ref{section:preliminary-steps}, adapting the
proof of \cite{Hei08} for 0-seminormality to work for
$p$-seminormality.\medskip
\par We start with the following two lemmas.

\begin{lemma}[{cf.\ \cite[Lemma 4 and Remark on p.\ 442]{Hei08}}]\label{step5}
  Let $\mathcal{P} =$ ``$p$-seminormal'', and let $(A,\fm)$ be our counterexample
  as in 
  Setup \ref{settingforthmsaandc} and Lemma \ref{step1}.
  Let $B = \tilde{B}$ as in Lemma \ref{step3}$(\ref{step3btilde})$.
If $b \in B$ is such that $b^m, b^{m + 1} \in A + yB$ for some integer $m > 0$ and $p b \in A+yB$, then $b \in A + yB$. Also, we have $\mathfrak{m} B = \mathfrak{m} + y B = J(B)$, where $J(B)$ is the Jacobson radical of $B$,
and $B/\mathfrak{m}B$ is a direct product of fields.
\end{lemma}
\begin{proof}
We begin with the first statement. Let $\bar{b}$ be the image of $b \in B / y B$. By the previous lemma, we can consider $\bar{b}$ as an element in $K = Q(A / yA)$. Then we have $\bar{b}^m, \bar{b}^{m + 1} \in (A + yB) / yB = A / yA$ and similarly $p\bar{b} \in A / yA$. Since $A / yA$ is $p$-seminormal (here we are using the equivalent definition of $p$-seminormality given in Proposition \ref{prop:hamannforpseminormal}), it follows that $b \in A + yB$.

We now show the second statement. First recall that $y\in \mathfrak{m}$ and so $yB\subseteq \mathfrak{m}B$. Since also $\mathfrak{m}\subseteq\mathfrak{m}B$, we then have $\mathfrak{m}+yB\subseteq \mathfrak{m}B$. Conversely, we claim that $\sqrt{\mathfrak{m}B} \subseteq \mathfrak{m} + yB$; this will show that $\mathfrak{m}B =\mathfrak{m} + yB = \sqrt{\mathfrak{m}B}$. We know that $\sqrt{\mathfrak{q}+yA}=\mathfrak{m}$ by Lemma \ref{step1}, and hence $\mathfrak{m}^k\subseteq\mathfrak{q}+yA$ for some $k>0$. Then for every $c \in \sqrt{\mathfrak{m}B}$, we have $c^m \in \mathfrak{m}^kB$ for some $m \geq k$, and hence
\begin{align*}
    c^m,c^{m+1}\in(\mathfrak{q}+yA)B = \mathfrak{q}B+yAB = \mathfrak{q}+yB \subseteq A+yB,
\end{align*}
where the second equality follows from Lemma \ref{step3}$(\ref{step3bmodfin})$.
Moreover, since $px \in A$, we have $p \in (A :_A x) = \fq$, and hence
$pc \in \fq B = \fq \subseteq A+yB$ by Lemma \ref{step3}$(\ref{step3bmodfin})$.
By the first statement shown above, we know that $c \in A + yB$, so we have $c = a + yr$ for some $a \in A$, and $r \in B$. Since $c, yr \in \mathfrak{m}B$, we find $a \in \mathfrak{m}B \cap A = \mathfrak{m}$, giving us $c \in \mathfrak{m} + yB$ as needed. 

Now consider an irredundant primary decomposition \cite[Theorem 7.13]{AM69} of $\mathfrak{m}B$:
\begin{align*}
    \mathfrak{m}B=\mathfrak{p}_1\cap \mathfrak{p}_2\cap\cdots \cap \mathfrak{p}_n.
\end{align*}
Since $\mathfrak{m}B$ is radical, all the $\mathfrak{p}_i$'s are prime. Furthermore, for each $i$ we have $\mathfrak{m}=\mathfrak{m}B\cap A\subseteq \mathfrak{p}_i\cap A$, so by \cite[Corollary 5.8]{AM69} we see that $\mathfrak{p}_i$ is maximal in $B$.
 Also since all maximal ideals of $B$ contract to $\mathfrak{m}$ we see that $\mathfrak{m}B$ is contained in all maximal ideals of $B$. Hence $\mathfrak{m}B$ is exactly the intersection of all maximal ideals of $B$, i.e.\ $\mathfrak{m}B=J(B)$.

Now that we have shown that $\mathfrak{m}B$ is the Jacobson radical, we have
\[\mathfrak{m}B = J(B) = \bigcap_{i=1}^l \mathfrak{n}_i \]
for the maximal ideals
$\mathfrak{n}_1,\mathfrak{n}_2,\ldots,\mathfrak{n}_l$ of $B$, of which there are
finitely many by Lemma \ref{step3}$(\ref{step3bmodfin})$.
Since the $\mathfrak{n}_i$ are comaximal, it follows from \cite[Proposition
1.10]{AM69} that
\[\frac{B}{\mathfrak{m}B} \cong \prod_{i = 1}^{l}
\frac{B}{\mathfrak{n}_{i}}.\qedhere\]
\end{proof}

\begin{lemma}[{cf.\ \cite[Lemma 5]{Hei08}}]\label{step6}
  Let $\mathcal{P} =$ ``$p$-seminormal'', and let $(A,\fm)$ be our counterexample
  as in 
  Setup \ref{settingforthmsaandc} and Lemma \ref{step1}.
  Let $B = \tilde{B}$ as in Lemma \ref{step3}$(\ref{step3btilde})$.
Suppose that $u, s \in B$, $\delta \in A$, and $xs \in A$ satsfiy a relation
\[xu = \delta + y^e s\]
for some $e > 0$.
Then, $s \in \mathfrak{m}B$.
\end{lemma}
\begin{proof}
Multiplying throughout by $x$, we have
\begin{align*}
    x^2u=\delta x + xy^es.
\end{align*}
Since $x^2\cdot x=x^3 \in A$ we have $x^2\in \mathfrak{q} = (A:_A x)$,
and since $u\in B$, we see that $x^2u\in \mathfrak{q}B=\mathfrak{q}\subseteq A$ by Lemma \ref{step3}$(\ref{step3bmodfin})$. Also, $xs\in A$ by assumption, so we must also have $\delta x\in A$. Hence $\delta\in \mathfrak{q}$ by definition of $\mathfrak{q}$ as $(A:_A x)$ (see Lemma \ref{step1}). Now since $y^es=xu-\delta$, for every integer $k>1$ we have
\begin{align*}
  (y^es)^k = (xu-\delta)^k = \sum_{i=0}^k\binom{k}{i}(xu)^i\delta^{k-i},
\end{align*}
which we claim lies in $\mathfrak{q}$.
For $i<k$ we have $\delta^{k-i}\in \mathfrak{q}$ and so $(xu)^i\delta^{k-i}\in B\mathfrak{q}=\mathfrak{q}B=\mathfrak{q}$ by Lemma \ref{step3}$(\ref{step3bmodfin})$.
For $i = k$, 
we have $x \in A'$ since $x^3 \in A$, and hence
$x \in B\coloneqq\Set{t \in A' \given \mathfrak{q}t \subseteq A}$. By assumption, $u \in B$, so $x^{k-2}u^k \in B$. Thus,
we have
\begin{align*}
    (xu)^k = (x^2) \cdot (x^{k-2}u^k)\in \mathfrak{q}B=\mathfrak{q}
\end{align*}
by Lemma \ref{step3}$(\ref{step3bmodfin})$.
\par It remains to show that $s \in \fm B$.
If $y^n s^k \in A$ for some $n$, then $y^n s^k \in yB \cap A = yA$ (by Lemma \ref{step4}$(\ref{step4:npsn})$), so $y^{n-1} s^k \in A$. Since $y^{ek}s^k \in A$ by the previous paragraph, downward induction implies that $s^k \in A$. If $y \in \mathfrak{q}$, then $yx \in (A \cap yB) - yA$, which contradicts Lemma \ref{step4}$(\ref{step4:npsn})$. Since $y \notin \mathfrak{q}$, we therefore have $s^k \in \mathfrak{q} \subseteq \mathfrak{m}B$. Since $\mathfrak{m}B$ is radical by Lemma \ref{step5}, this implies $s \in \mathfrak{m}B$, as desired.
\end{proof}

Theorem \ref{thm:pseminormallifts}$(\ref{thm:pseminormalliftsfixedp})$ now
follows as in \cite[pp.\ 442--444]{Hei08}, replacing the preliminary lemmas in
\cite{Hei08} with our preliminary results above in
\S\ref{section:preliminary-steps} and in this section.
For completeness, we write down the proof below.

\begin{proof}[Proof of Theorem
\ref{thm:pseminormallifts}$(\ref{thm:pseminormalliftsfixedp})$]
Let $B = \tilde{B}$ as in Lemmas \ref{step5} and \ref{step6}.
We start by setting some notation.
Let 
\[B_i \coloneqq A + x\mathfrak{m}B + y^iB \subseteq B.\]
This gives a descending chain of $A$-modules 
\[B \supseteq B_1 \supseteq B_2 \supseteq \cdots \supseteq A + x\mathfrak{m}B.\]
By Krull's intersection theorem \cite[Exercise 18.35]{AK21}, we then have
\[A + x\mathfrak{m}B \subseteq \bigcap_i B_i \subseteq \bigcap_i (A + x\mathfrak{m}B + \mathfrak{m}^iB) = A + x\mathfrak{m}B,\]
and hence $\bigcap_i B_i = A + x\mathfrak{m}B$.

Now set $U_i = \Set{t \in B \given xt \in B_i}$. From the above, we have
\[U\coloneqq \bigcap_i U_i = \Set[\big]{t\in B\given xt \in A + x\mathfrak{m}B}.\]
Using the descending chain of $B_i$'s, we have another descending chain of
$A$-modules:
\[B \supseteq U_1 \supseteq U_2 \supseteq \cdots \supseteq U.\]
Moreover, for every $t \in B$, we have
$x^2 t^2, x^3 t^3, p(xt) \in \fq B = \fq \subseteq A$
by Lemma \ref{step3}$(\ref{step3bmodfin})$ since $x^2, x^3, p
\in \mathfrak{q}$.
Thus, $xt \in A + yB \subseteq B_1$ by Lemma \ref{step5}, and hence $U_1 = B$. Now we claim that $U = U_m$ for some $m$. First, we note that $\mathfrak{m}B \subseteq U$. It follows that $B / U$ is a finite-dimensional $A / \mathfrak{m}$-vector space, and hence the chain of submodules $B / U \supseteq U_1 / U \supseteq U_2 / U \supseteq \cdots \supseteq U / U = 0$ stabilizes. Therefore, $U = U_m$ for some $m$.\smallskip

\par We now proceed in a sequence of steps.

\begin{step}
  Choosing a sequence of generators $u_1,u_2,\ldots,u_n$ for $B$ lifting a basis
  for $\bar{B} \coloneqq B/\mathfrak{m}B$.
\end{step}

Consider the homomorphism $B \to \bar{B} \coloneqq B / \mathfrak{m}B$, and denote by $\bar{U}$ and $\bar{U}_i$ the images of $U$ and $U_i$ under this homomorphism. We then have an ascending chain
\[ \bar{U} = \bar{U}_m \subseteq \bar{U}_{m - 1} \subseteq \cdots \subseteq \bar{U}_1 = \bar{B}. \]

Now choose a basis for the $A/\mathfrak{m}$-vector space $\Bar{B}$ that contains
a basis for $\Bar{U}_i$ for each $i$. By \cite[Proposition 2.8]{AM69}, we can
then lift this basis to a generating set for $B$ in the following manner:
Let $\Bar{b}$ be an element of this basis, which lifts to some $b' \in B$. Given that $\mathfrak{m}B \subseteq U$, whether or not $b' \in U$ is independent of the lifting. So if $b' \in U$, then for some $r \in A$ and $t \in \mathfrak{m}B$, we have 
\[xb' = r + xt \in A + x\mathfrak{m}B.\]
Setting $b = b' - t$, we then get 
\[xb = x(b' - t) = r + xt - xt = r \in A.\]
If instead, we have $b' \notin U$, then let $j$ be the largest integer such that $b' \in U_j$. Again, since $b'$ is a lift of our chosen basis of $\Bar{B}$, the number $j$ is independent of the lifting. So as before, we have
\[xb' = r + xt + y^ja \in A + x\mathfrak{m}B + y^j B.\]
for some $r \in A, t \in \mathfrak{m}B$, and $a \in B$. Setting $b = b'-t$, we then have
\[xb = r + y^j a \in A + y^j B.\]

Let $\dim {\Bar{B}} = n$. By the procedure above, we have a generating set of $B$ of size $n$. We enumerate the elements in this generating set $u_1, u_2, \ldots, u_n$ for $B$, so that if $\dim_{A/\mathfrak{m}}(\Bar{U}_j) = n-k_j > 0$, then $\Bar{u}_{k_j+1}, \Bar{u}_{k_j+2}, \ldots, \Bar{u}_n$ is a basis for $\Bar{U}_j$. In particular, if $\dim_{A/\mathfrak{m}}(\Bar{U}) = n - k \geq 0$, then $\Bar{u}_{k_j+1}, \Bar{u}_{k_j+2}, \ldots, \Bar{u}_n$ is a basis for $\Bar{U}$.\smallskip

In Step \ref{step9} below, we use the following notation.
For each $i\le k$, we have a generator $u_i$ for which $u_i\in U_{e_i}-U_{e_i+1}$ for some $e_i\in\{1,2,\dots,m\}$. For this element $u_i$, we may write
\begin{align}
    xu_i = \alpha_i + y^{e_i}s_i \label{xuiRelation}
\end{align}
for some $\alpha_i\in A$ and $s_i\in B$. Suppose that $s_i\in A+yB$, say $s_i=r+yb$ for some $r\in A,b\in B$. Then 
\begin{align*}
    xu_i &= \alpha_i+y^{e_i}(r+yb)
    = \alpha_i + y^{e_i}r + y^{e_i+1}b \in U_{e_i+1}
\end{align*}
This is a contradiction, so $s_i\not\in A+yB$. Hence $s_i\not\in \mathfrak{m}B$ by Lemma \ref{step5}. Repeating this process for every $i\leq k$, we obtain a sequence of elements $s_1,s_2,\dots,s_k \in B - \fm B$. Note that we may freely add elements of $A$ to $s_i$ without changing the property that \eqref{xuiRelation} holds for some $\alpha_i\in A$.

\begin{step}\label{step9}
The subset $\{\bar{1},\bar{s}_1,\bar{s}_2,\dots,\bar{s}_k\}\subseteq \bar{B}$ is linearly independent, and hence if $C\coloneqq A+\sum_{i=1}^ks_iA$, its image $\bar{C}$ in $\bar{B}$ spans a subspace of dimension $k+1$.
\end{step}
If this is not the case, pick $j$ minimally such that the set $\{\bar 1, \bar s_1, \ldots, \bar s_j\}$ is linearly dependent, and pick $\rho_i \in A$ such that
\[s_j - \sum_{i < j} \rho_i s_i \in A +\mathfrak{m}B.\]
Let $f_i = e_j - e_i \geq 0$ for $i \leq j$, and let
\[
    u = u_j - \sum_{i < j} y^{f_i} \rho_i u_i.
\]
We then have
\begin{align*}
    xu &= xu_j-\sum_{i<j}y^{f_i}\rho_ixu_i\\
    &= \alpha_j + y^{e_j}s_j - \sum_{i<j}y^{f_i}\rho_i(\alpha_i+y^{e_i}s_i)\\
    &= \left(\alpha_j-\sum_{i<j}y^{f_i}\rho_i\alpha_i\right)+y^{e_j}\left(s_j-\sum_{i<j}\rho_is_i\right)
\end{align*}
and hence
\begin{align*}
    xu \in A+y^{e_j}(A+\mathfrak{m}B) = A+y^{e_j}(A+yB) = A+y^{e_j+1}B
\end{align*}
since $\mathfrak{m}B=\mathfrak{m}+yB$ by Lemma \ref{step5}. However, this implies that $u\in U_{e_j+1}$, in which case $\bar{u}_1,\bar{u}_2,\dots,\bar{u}_j$ are not linearly independent modulo $\bar{U}_{e_j+1}$, contradicting the choice of generating set.\smallskip

In the rest of the proof, we use the following notation.
Let $T = \Set{t \in B \given xt \in A}$, and let $\bar T$ be the image of $T$ in $\bar B$. By definition, $T \subseteq U$, and we have $u_{k + 1}, u_{k+2},\ldots, u_n \in T$, so $\bar T = \bar U$, and $\dim_{A/\fm}(\bar T) = n - k$. Since $\bar C + \bar T \subseteq \bar B$, we have
\[\dim_{A/\fm}(\bar C + \bar T) = \dim_{A/\fm}(\bar B) - d = n - d\]
for some $d \geq 0$. Therefore, we have:
\begin{align*}
    \dim_{A/\fm}(\bar C \cap \bar T) &= \dim_{A/\fm}(\bar C) + \dim_{A/\fm}(\bar T) - \dim_{A/\fm}(\bar C + \bar T)\\
    &= (k + 1) + (n - k) - (n - d)
    = d + 1 \geq 0,
\end{align*}
so $\bar C \cap \bar T \neq 0$. Our goal is to find an appropriate element $s \in C \cap T$ that lifts an element in this intersection, which we will use to produce a contradiction.

\begin{step}\label{step10}
We have $\bar{1}\not\in\bar{T}$.
\end{step}

We first show that $T = (\fq :_B x)$. The inclusion $T \supseteq (\fq:_B x)$ follows from the fact that $T = (A :_B x)$. Conversely, if $t \in T$, then $xt \in A$, so
\[ (xt)^2 = x^2 t^2 \in \fq B = \fq \]
by Lemma \ref{step3}$(\ref{step3bmodfin})$ (here, we again use the fact that $x^2 \in \fq$, which holds since $x^2x = x^3 \in A$ implies $x^2 \in \fq = (A:_A x)$). Now, if $\bar{1} \in \bar{T}$, then $T + \fm B = B$, so $T = B$ by the NAK lemma \cite[Corollary 2.7]{AM69}. This implies that $1 \in T$, and hence $x \in A$, a contradiction. Therefore, we must have $\bar{1} \notin \bar{T}$.

\begin{step}\label{step11}
Choose elements $r_j \in C$ for $j \in \{1,2,\cdots, d+1\}$ that map to a basis of $\Bar{C} \cap \Bar{T}$. Let 
\[E\coloneqq \sum^{d+1}_{j=1} r_j A \subseteq C,\]
whose image $\Bar{E}$ in $\Bar{B}$ satisfies $\Bar{E} = \Bar{C} \cap \Bar{T}$ by our choice of $r_j$'s. Then, there exists an element
\[z \in \bigl((T + \fm + yC) \cap E\bigr) - \fm B.\]
\end{step}

We first show there exists a subset $F = \{z_1, z_2, \ldots, z_g\}$ of $E$ that satisfies the following properties:

\begin{enumerate}[label=$(\arabic*)$,ref=\arabic*]
    \item\label{step11prop1} The set $\{\Bar{z}_1, \Bar{z}_2, \ldots, \Bar{z}_g\}$ is a linearly independent subset of $B$.
    \item\label{step11prop2} For every $i \in \{1,2\ldots,g\}$, we have $z_i = b_i + y^{e_i}t_i$, where $b_i \in T + \fm + yC$, $e_i \in \mathbb{Z}_{> 0}$, $t_i \in B$, and $\{\Bar{\Bar{t}}_1, \Bar{\Bar{t}}_2, \ldots, \Bar{\Bar{t}}_g\}$ is a linearly independent subset of $\Bar{B}/(\Bar{C}+\Bar{T}).$
    \item\label{step11prop3} The number $g \geq 0$ is maximal for sets satisfying $(\ref{step11prop1})$ and $(\ref{step11prop2})$.
    \item\label{step11prop4} The number $\sum_i e_i$ is minimal among sets
      satisfying $(\ref{step11prop1})$, $(\ref{step11prop2})$, and
      $(\ref{step11prop3})$. 
    \item\label{step11prop5} $e_1 \leq e_2 \leq \ldots \leq e_g$.
\end{enumerate}
The collection $P$ of all subsets $F$ satisfying $(\ref{step11prop1})$ and $(\ref{step11prop2})$ can be partially ordered by inclusion, and $P\ne\emptyset$ since it contains the empty set. Since $g\le\dim_{A/\fm}(\bar{B})$, we can restrict to a subcollection of $P$ whose cardinality $g$ is maximal, ensuring $(\ref{step11prop3})$. Finally, we pick any set in this subcollection with minimial $\sum_ie_i$ to ensure $(\ref{step11prop4})$, and reorder the elements if necessary so that $(\ref{step11prop5})$ holds. We have
\begin{align*}
    \dim_{R/\fm}\left(\frac{\bar{B}}{\bar{C}+\bar{T}}\right)=n-(n-d)=d
\end{align*}
and so $g\le d<d+1=\dim_{A/\fm}(\bar{E})$. Thus, $\{\bar{z}_1,\bar{z}_2,\dots,\bar{z}_g\}$ has too few elements to span all of $\bar{E}$.

We claim that we can find a sequence $v_1,v_2,\ldots$ of elements in $E$ such
that for all $i$, the set $\{\Bar{z}_1,\Bar{z}_2,\ldots,\Bar{z}_g,\Bar{v}_i\}$ is linearly independent in $\Bar{E}$ and 
\[v_i \in T+\fm+yC+y^iB\]
when $v_{i+1} - v_i \in y^{i-e_g}B$ for all $i > e_g$. Here, we set $e_0 = 0$.
We proceed by induction on $i$. For $i =1$, choose $v_1 \in E$ such that
$\{\Bar{z}_1,\Bar{z}_2,\ldots,\Bar{z}_g,\Bar{v}_1\}$ is a linearly independent subset of $\Bar{E}$.
This yields $v_1 \in T+\fm B = T+\fm+yB$ as desired, with the equality holding
by the result of Lemma \ref{step5}. For the inductive step, say we have already picked $v_j \in E$, which we can write as
\[v_j = a + y^jt\]
for $a \in T+\fm+yC$ and $t\in B$. By the maximality of $g$ and the choice of
$F$, the set 
\[\bigl\{\Bar{\Bar{t}}_1,\Bar{\Bar{t}}_2,\ldots,\Bar{\Bar{t}}_g,\Bar{\Bar{t}}\bigr\} \subseteq \frac{\Bar{B}}{\Bar{C}+\Bar{T}}\]
cannot be linearly independent. Furthermore, by the minimality of $\sum_i e_i$, the set 
\[\bigl\{\Bar{\Bar{t}}_1,\Bar{\Bar{t}}_2,\ldots,\Bar{\Bar{t}}_h,\Bar{\Bar{t}}\bigr\} \subseteq \frac{\Bar{B}}{\Bar{C}+\Bar{T}}\]
is linearly dependent when $e_{h+1}>j$. Therefore, we may write 
\[\Bar{\Bar{t}} = \sum_{i\leq h}\Bar{\Bar{\rho}}_i\Bar{\Bar{t}}_i\]
for some $\rho_i\in R$. This gives 
\[t - \sum_{i\leq h}\rho_it_i \in C+T+\fm B = C+T+yB,\]
where the rightmost equality holds by the definition of $C$ in Step \ref{step9}
and since $\fm B = \fm + yB \subseteq C + yB$ by Lemma \ref{step5}. By writing
\[
    t-\sum_{i\le h} \rho_it_i=c+yb
\]
for $c\in C+T$ and $b\in B$, we have
\[
    v_j = a+y^jt = (a+y^jc)+y^j(t-c),
\]
and hence we can replace $t$ by $t-c$ and $a$ by $a+y^jc$ to assume that
$t=\sum_{i\le h}\rho_it_i\in yB$. Now let $f_i\coloneqq j-e_i$ and set
\[
    v_{j+1} = v_j - \sum_{i\le h}y^{f_i}\rho_iz_i
    = \biggl(a-\sum_{i\le h}y^{f_i}\rho_ib_i\biggr)+y^j\biggl(t-\sum_{i\le
    h}\rho_it_i\biggr),
\]
in which case
\[
  v_{j+1}\in T+\fm+yC+y^j(yB) = T+\fm+yC+y^{j+1}B.
\]
Since $f_i \geq j-e_g$ for all $i$, we have
\[
  v_{j+1} - v_j = \sum_{i\leq h} y^{f_i}\rho_i z_i \in y^{j-e_g}B.
\]
We have now constructed the desired sequence $v_1,v_2,\ldots$ of elements in
$E$.

Finally, since $A$ is $y$-adically complete by Lemma \ref{step2}, the finitely generated $A$-module $B$ is also $y$-adically complete \cite[Theorem 10.13]{AM69}. Thus, the Cauchy sequence $v_1, v_2, \ldots$ has a limit in $B$, which we denote by $z$. We then have that
\[z \in (T+\fm+yC+y^iB)\cap(E+y^iB)\]
for every $i$. By Krull's intersection theorem \cite[Exercise 18.35]{AK21}, we therefore have
\[z \in (T+\fm+yC)\cap E.\]
Finally, $\Bar{z} = \Bar{v}_{e_g+1} \ne 0$ implies $z \not\in \fm B$. 

\begin{step}\label{step12}
Conclusion of the proof
\end{step}

Write $z = s + c$ for $s \in T$ and $c \in \fm + yC$. We have $z \in E - \fm B \subseteq C - \fm B$. Since $c \in C \cap \fm B$, we have $s \in (C \cap T) - \fm B$. Since $\bar 1 \notin \bar T$ by Step \ref{step10}, we have $\bar s \notin A/\fm \subseteq \bar B$. Therefore, we have
\[s = \gamma + \sum_{i \leq k} \gamma_i s_i\]
for some $\gamma, \gamma_i \in A$, where the $\gamma_i$ are not all in $\fm$ (otherwise we would have $\bar{s} = 0 \in A/\fm$).
Then, $\gamma_j$ is a unit for some $j$, so we may replace $s_j$ by $s_j + \gamma/\gamma_j$ to get
\[s = \sum_{i \leq k} \gamma_i s_i.\]
(Recall that adding an element of $A$ to the $s_i$'s does not change the fact that
\eqref{xuiRelation} holds for some $\alpha_i \in A$.) The linear independence of
the $\bar s_i$'s is also preserved, since $\bar 1 \notin \sum_{i \leq k} (A/\fm)
s_i$ by Step \ref{step9}.

Now recall that \eqref{xuiRelation} holds for all $i \leq k$. Let $f_i \coloneqq e_k - e_i \geq 0$, and set
\[ u = \sum_{i \leq k} y^{f_i} \gamma_i u_i \in B. \]
Then, we have
\[
    xu = \sum_{i \leq k} y^{f_i} \gamma_i (\alpha_i + y^{e_i} s_i) 
    = \sum_{i \leq k} y^{f_i} \gamma_i \alpha_i + \sum_{i \leq k} y^{e_k - e_i} y^{e_i} \gamma_i s_i 
    = \sum_{i \leq k} y^{f_i} \gamma_i \alpha_i + y^{e_k} s 
    = \delta + y^{e_k} s,
\]
where $\delta \coloneqq \sum_{i \leq k} y^{f_i} \gamma_i \alpha_i \in A$. We
also know that $xs \in A$ since $s \in T$. On the other hand, by Lemma
\ref{step6},
this implies that $s \in \fm B$, a contradiction.
\end{proof}

\section{Normality and weak normality deform}\label{section:nwnlift}
We now give our new proofs of the facts that normality and weak normality deform.
\begin{customthm}{C}\label{thm:nwnlift}
  Let $(A,\fm)$ be a Noetherian local ring, and let $y \in \fm$ be a
  nonzerodivisor.
  If $A/yA$ is normal (resp.\ weakly normal), then $A$ is normal (resp.\ weakly
  normal).
\end{customthm}

\begin{proof}
Suppose that $A$ is not normal (resp.\ weakly normal).
As in Setup \ref{settingforthmsaandc} (which uses \cite[Main Theorem]{Hei08} or
Theorem \ref{thm:pseminormallifts} in the weakly normal case), there exists an
element $x \in Q(A) - A$ that is integral over $A$ (resp.\ such that $x^p,px \in
A$ for some prime number $p > 0$).
After applying the reduction in Lemma \ref{step1}, we may assume that
the ideal $\fq \coloneqq (A :_A x)$ is prime.
Letting $B = \tilde{B} \coloneqq \Set{t \in A' \given \fq t \subseteq A}$
(resp.\ $B = \Set{t \in \tilde{B} \given t^p \in A}$),
we have the commutative diagram
\[
\begin{tikzcd}
A/yA \arrow[r,hook]        & B/yB \arrow[r,hook] & Q(A/yA) \\
A \arrow[u] \arrow[r,hook] & B \arrow[u]            &        
\end{tikzcd}
\]
where the horizontal arrows are injective by Lemma
\ref{step4}$(\ref{step4:npsn})$ (resp.\ Lemma
\ref{step4}$(\ref{step4:wn})$).
Note that $x \in B$ by the definition of $B$.
\par We claim that $A \to B$ is an isomorphism.
Since this map is injective by the above, it suffices to show that every $b \in
B$ lies in $A$.
\par We first consider the normal case.
Let $f \in A[T]$ be a monic polynomial such that $f(b) = 0$, and let $\bar{f}$
be its reduction modulo $y$. Then, $\bar{f} (\bar{b}) = 0$, so $\bar{b}$ is integral over $A / yA$. Since $\bar{b}$ lies in $Q(A/yA)$, it follows that $\bar{b} \in A / yA$, and hence $A / yA = B / yB$. Therefore, we have $B = A + yB$, so $A = B$ by the NAK lemma \cite[Corollary 2.7]{AM69}.
This is a contradiction, since $x \in B - A$ by construction.
\par Finally, we consider the weakly normal case.
By the definition of $B$, we have $b^p \in A$ and $p b \in \mathfrak{q} B \subseteq A$, so $\bar{b}^p, p\bar{b} \in A/yA$. Since $A / yA$ is weakly normal, it follows that $\bar{b} \in A / yA$, and hence $A / yA = B / yB$. Therefore, we have $B = A + yB$, so $A = B$ by the NAK lemma.
This is a contradiction, since $x \in B - A$ by construction.
\end{proof}

\end{document}